\documentclass[12pt, a4paper]{amsart}
 
\usepackage{amssymb,enumerate}
\usepackage{amsmath,tabu}
\usepackage{thmtools}

\usepackage{amsthm}
\usepackage[utf8]{inputenc} 
\usepackage[T1]{fontenc}
\usepackage{mathrsfs}
\usepackage{caption}
\usepackage{enumitem}
\usepackage[toc,page]{appendix}
\usepackage{tabu}
\usepackage{longtable}
\usepackage{tikz-cd}

\setlist[enumerate,1]{label=(\alph*)}
\setlist[enumerate,2]{label=(\roman*), ref=\theenumii.\roman*}
\begin{document}

\newtheorem{thm}{Theorem}[section]
\newtheorem{lem}[thm]{Lemma}
\newtheorem{prop}[thm]{Proposition}
\newtheorem{cor}[thm]{Corollary}
\newtheorem*{conjecture}{Conjecture}
\theoremstyle{remark}
\newtheorem{remark}[thm]{Remark}
\theoremstyle{definition}
\newtheorem{assumption}[thm]{Assumption}
\newtheorem{definition}[thm]{Definition}
\newtheorem{example}[thm]{Example}

\title{On 
$e$-cuspidal pairs of finite groups of exceptional Lie Type}

\date{\today}

\author{Ruwen Hollenbach}
\address{IAZD, Leibniz Universität Hannover,
         30167 Hannover, Germany.}
\email{hollenbach@math.uni-hannover.de}



\begin{abstract}
Let $G$ be a simple, simply connected algebraic group of exceptional type defined over $\mathbb{F}_q$ with Frobenius endomorphism $F: G \to G$. Let $\ell \nmid q$ be a good prime for $G$. We determine the number of irreducible Brauer characters in the quasi-isolated $\ell$-blocks of $G^F$. This is done by proving that generalized $e$-Harish-Chandra theory holds for the Lusztig series associated to quasi-isolated elements of $G^{*F}$.

\end{abstract}

\maketitle

\section{Introduction}
\noindent

The results in this paper are part of the author's PhD thesis \cite{HollenbachDiss}. The underlying goal was and still is to show that finite quasi-simple groups do not yield minimal counterexamples to a conjecture that was proposed recently by Malle and Robinson \cite{Malle-Robinson} (referred to as the Malle--Robinson conjecture here). Let $G$ be a finite group and let $\ell$ be a prime dividing the order of $G$. They proposed the following.
\begin{conjecture}[Malle--Robinson, {\cite[Conjecture 1]{Malle-Robinson}}]\label{Malle-Robinson Conjecture}
Let $B$ be an $\ell$-block of $G$ with defect group $D$. Then 
\begin{align*}
l(B) \leq \ell^{s(D)},
\end{align*}
where $l(B)$ denotes the number of irreducible Brauer characters in $B$ and  $s(D)$ denotes the sectional $\ell$-rank of $D$.
\end{conjecture}
In \cite{Malle-Robinson} this conjecture was proved to hold for the blocks of $p$-solvable groups and some quasi-simple finite groups. However, not much is known in the case of finite groups of Lie type in cross-characteristic.
In this article we prove the Malle--Robinson conjecture for  the so-called quasi-isolated $\ell$-blocks (see Section 4) of the finite quasi-simple groups of exceptional Lie type in cross characteristic. The proof involves steps that should be of independent interest. The reason we focus on quasi-isolated blocks is the reduction by Bonnafé-Dat-Rouquier \cite{BDR} of block theoretic questions for groups of Lie type to a quasi-isolated setting.
 \\

Let $G$ be a simple, simply connected algebraic group of exceptional type defined over $\mathbb{F}_q$ with Frobenius endomorphism $F: G \to G$, or let $G$ be a simple, simply connected algebraic group of type $D_4$ defined over $\mathbb{F}_q$ with Frobenius endomorphism $F:G \to G$ such that $G^F=\text{} ^3\!D_4(q)$, where $G^F:=\{g \in G \, \mid \, F(g)=g\}$. It is well-known that then $G^F/Z$, for any $Z \leq Z(G^F)$, is a finite quasi-simple group.

Now, let $\ell \nmid q$ be a good prime for $G$ and further assume that $\ell \neq 3$ if $G^F=\text{}^3\!D_4(q)$.  \\

We start by determining the number of quasi-isolated $\ell$-blocks of $G^F$ and their general structure. By results of Cabanes--Enguehard \cite{Cabanes-Enguehardtwisted}, there is a bijection between the blocks of $G^{F}$ and the $G^F$-conjugacy classes of its $e$-cuspidal pairs. Thus it remains to determine the $G^{F}$-conjugacy classes of $e$-cuspidal pairs corresponding to the quasi-isolated blocks. These pairs and some additional information that is needed in Section 4 can be found in the tables in Section 3. Next, in order to determine $l(B)$ for a given quasi-isolated $\ell$-block $B$ we prove that \textit{generalized $e$-Harish-Chandra theory} (see Section 2) holds in the Lusztig series associated to $B$ (by Theorem \ref{Lusztigseries blocks}). 
This step is carried out case-by-case for the different exceptional types in Section 3 and culminates in the following.
\begin{restatable}{thm}{A}\label{e-Harish-Chandra good prime} Let $G$, $F$, and $\ell$ be as above. If $s \in G^{*F}$ is a semisimple, quasi-isolated $\ell'$-element then generalized $e$-Harish-Chandra theory holds in $\mathcal{E}(G^F,s)$.
\end{restatable}
\noindent
Note that the assertion of Theorem \ref{e-Harish-Chandra good prime} is known to hold when $G$ is of type $F_4$ or $E_8$ and $q>2$ by \cite{EnguehardJD} and  when $e=1$ or $2$ for all finite groups of exceptional Lie type by \cite{Malle-Kessar}. \\
 
In Section 4 we then  prove the Malle--Robinson conjecture for the quasi-isolated $\ell$-blocks of $G^F$. Let $B$ be a quasi-isolated $\ell$-block of $G^F$ with defect group $D$. Theorem 1.1 enables us to determine an explicit basic set for $B$ as follows. By \cite{Cabanes-Enguehardtwisted}, $B$ is parametrized by an $e$-cuspidal pair $(L, \lambda)$ and $\operatorname{Irr}(B)$ contains the set $\mathcal{E}(G^F,(L, \lambda))$ of irreducible constituents of $R_L^G(\lambda)$ (see Section 2). An immediate consequnce of Theorem \ref{e-Harish-Chandra good prime} is that $\mathcal{E}(G^F, (L, \lambda))$ is a basic set for $B$. Hence, $l(B)=|\mathcal{E}(G^F,(L, \lambda))|$. Furthermore, by \cite{Cabanes-Enguehardtwisted} we have $s(Z(L)_\ell^F) \leq s(D)$ where the left side can be determined using the tables in Section 3. In fact, it turns out that $s(Z(L)_\ell^F)$ is enough to verify the Malle--Robinson conjecture for $B$ and  we have the stronger inequality 
\[l(B) \leq s^{s(L)_\ell^F}. \]
Standard Clifford theoretic arguments then imply the Malle--Robinson conjecture for the quasi-isolated blocks of $G^F/Z(G^F)$.

\begin{restatable}{thm}{B} \label{Conjecture good}
Let $G$, $F$, and $\ell$ be as in Theorem \ref{e-Harish-Chandra good prime}. Let $B=b_{G^F}(L,\lambda)$ be a quasi-isolated $\ell$-block of $G^F$. Then $\mathcal{E}(G^F, (L, \lambda))$ is an ordinary basic set for $B$. Moreover, the Malle--Robinson conjecture holds for all quasi-isolated $\ell$-blocks of $G^F$ and of $G^F/Z(G^F)$.
\end{restatable}

Applying \cite[Theorem 7.7]{BDR} then yields the following corollary to Theorem \ref{Conjecture good}.

\begin{restatable}{cor}{C} \label{minimal}
Let $H$ be a finite quasi-simple group of exceptional Lie type. Let $\ell$ be a prime and let $B$ be an $\ell$-block of $H$. Then $B$ is not a minimal counterexample to the Malle--Robinson conjecture for $\ell > 5$ if $H=E_8(q)$ and for $\ell \geq 5$ otherwise.
\end{restatable}

\textit{Acknowledgements.} I would like to thank Jay Taylor for helpful remarks on an earlier version and pointing out some mistakes in the previous tables. I would also like to thank Gunter Malle for his constant support over the course of my PhD and his supervision on this project.

\section{Generalized $e$-Harish-Chandra theory}
Let $G$ be a connected reductive group defined over $\mathbb{F}_q$ with Frobenius endomorphism $F$. We start by recalling some relevant notions and results. Let $e \geq 1$ be an integer. Suppose that $S$ is an $F$-stable torus of $G$ with complete root datum $\mathbb{S}$ (see \cite[Definition 22.10]{Malle-Testerman}). Then $S$ is called an \textbf{$e$-torus} if $|\mathbb{S}|=\Phi_e(X)^a$ for some non-negative integer $a$, where $|\mathbb{S}|$ denotes the order polynomial of $\mathbb{S}$ and $\Phi_e$ denotes the $e$-th cyclotomic polynomial. A Levi subgroup $L$ of $G$ is called \textbf{$e$-split} if $L=C_{G}(S)$ is the centralizer of an $e$-torus $S$ of $G$.

\begin{prop}\label{e-split Levi}
Let $G$ be a connected reductive group defined over $\mathbb{F}_q$ with Frobenius endomorphism $F:G \to G$. If $L$ is an $e$-split Levi subgroup of $G$, then $L=C_{G}(Z^\circ(L)_{\Phi_e})$, where $Z^\circ(L)_{\Phi_e}$ denotes the $\Phi_e$-part of the torus $Z^\circ(L)$.
\end{prop}

\begin{proof}
Since $L$ is $e$-split, there exists an $e$-torus $S$ of $G$ such that $L=C_G(S)$. Clearly, $S \subseteq Z^\circ(L)_{\Phi_e}$. Since $L=C_G(Z^\circ(L))$ (see \cite[1.21 Proposition]{Digne-Michel}), we have
\begin{align*}
L=C_G(S) \supseteq C_G(Z^\circ(L)_{\Phi_e}) \supseteq C_G(Z^\circ(L))=L.
\end{align*}
Hence, $L=C_G(Z^\circ(L)_{\Phi_e})$.
\end{proof}

Let $R_{L \subseteq P}^G$ denote Lusztig induction from an $F$-stable Levi subgroup $L$ contained in a parabolic subgroup $P \subseteq G$ to $G$ and let $^*\!R_{L \subseteq P}^G$ denote Lusztig restriction (see \cite[11.1 Definition]{Digne-Michel}). We say that an irreducible character $\chi$ of $G^F$ is \textbf{$e$-cuspidal} if $^*\!R_{L \subseteq P}^G(\chi)=0$ for every $e$-split Levi subgroup $L$ contained in a proper parabolic subgroup $P \subseteq G$. Let $\lambda \in \operatorname{Irr}(L^F)$ for an $e$-split Levi subgroup $L \subseteq G$. Then we call $(L, \lambda)$ an \textbf{$e$-split pair}. We define a binary relation on $e$-split pairs by setting $(M, \zeta) \leq_e (L, \lambda)$ if $M \subseteq L$ and $\langle ^*\!R_{M \subseteq Q}^L(\lambda), \zeta \rangle \neq 0$. Since the Lusztig restriction of a character is in general not a character, but a generalized character, the relation $\leq_e$ might not be transitive. We denote the transitive closure of $\leq_e$ by $\ll_e$. If $(L, \lambda)$ is minimal for the partial order $\ll_e$, we call $(L, \lambda)$ an \textbf{$e$-cuspidal pair} of $G^F$. Moreover, we say $(L, \lambda)$ is a \textbf{proper} $e$-cuspidal pair if $L \subsetneq G$ is a proper $F$-stable Levi subgroup of $G$.  \\

Let $G^*$ be a group in duality with $G$ with respect to an $F$-stable maximal torus $T$ of $G$ (see \cite[13.10 Definition]{Digne-Michel}). By results of Lusztig, $\operatorname{Irr}(G^F)$ is a disjoint union of so-called (rational) Lusztig series $\mathcal{E}(G^F,s)$, where $s$ runs over the $G^{*F}$-conjugacy classes of semisimple elements of the dual group $G^*$ (see \cite[p. 138]{Digne-Michel}). 

The following definition can be found in \cite[2.2.1 Definition]{Enguehard}. Let $s \in G^{*F}$ be semisimple. For an $e$-split pair $(L, \lambda)$ we set $N_{G^F}(L, \lambda):= \{ g \in N_{G^F}(L) \, \mid \, \lambda^g = \lambda \}$ and call $W_{G^F}(L, \lambda):= N_{G^F}(L, \lambda)/L^F$ its \textbf{relative Weyl group}. We say that \textbf{generalized $e$-Harish-Chandra theory} holds in $\mathcal{E}(G^F,s)$ if, for any $\chi \in \mathcal{E}(G^F,s)$ there exists an $e$-cuspidal pair $(L, \lambda)$ of $G^F$, uniquely defined up to $G^F$-conjugacy, and an integer $a \neq 0$ such that
\begin{align*}
^*\!R_{L \subseteq P}^G \chi = a \left(\sum_{g \in N_{G^F}(L)/N_{G^F}(L, \lambda)} \lambda^g \right)
\end{align*}
for every parabolic subgroup $P \subseteq G$ containing $L$. \\

Recall the following classical result about the block theory of finite groups of Lie type.

\begin{thm}[{\cite[2.2 Théorème]{Broue}, \cite[Theorem 3.1]{Hiss}}]\label{Lusztigseries blocks}
Let $s \in G^{*F}$ be a semisimple $\ell'$-element. Then we have the following.\\
\begin{tabular}{rp{13,7cm}}
(a) & The set $
\mathcal{E}_\ell(G^F,s):= \bigcup_{t \in C_{G^*}(s)_\ell^F} \mathcal{E}(G^F,st)$
 is a union of $\ell$-blocks of $G^F$. \\
(b) & Any $\ell$-block contained in $\mathcal{E}_\ell(G^F,s)$ contains a character of $\mathcal{E}(G^F,s)$.
\end{tabular}
\end{thm}
\noindent
Now, $e$-cuspidal pairs yield a refinement  of the result above and enable us to paraemtrize the blocks occuring in the union in Theorem \ref{Lusztigseries blocks} (a).  Let $e_\ell(q)$ denote the multiplicative order of $q$ modulo $\ell$.

\begin{thm}[{\cite[Theorem 4.1]{Cabanes-Enguehardtwisted}}]\label{blocks good l}
Let $G$ be a connected reductive group defined over $\mathbb{F}_q$ with Frobenius endomorphism $F:G \to G$. Let $\ell$ be a good prime for $G$ not dividing $q$. Furthermore, assume that $\ell \neq  3$ if $G^F$ has a component of type $^3\!D_4(q)$. Let $s \in G^{*F}$ be a semisimple $\ell'$-element. If $e=e_\ell(q)$, then we have the following. 
\begin{enumerate}
\item There is a natural bijection
\begin{align*}
b_{G^F}(L,\lambda) \longleftrightarrow (L,\lambda)
\end{align*}
between the $\ell$-blocks of $G^F$ contained in $\mathcal{E}_\ell(G^F,s)$ and the $e$-cuspidal pairs $(L, \lambda)$, up to $G^F$-conjugation, such that $s \in L^{*F}$ and $\lambda \in \mathcal{E}(L^F,s)$, where $b_{G^F}(L, \lambda)$ is the unique block containing the irreducible constituents of $R_L^G(\lambda)$.
\item If $B=b_{G^F}(L, \lambda)$, then $\operatorname{Irr}(B) \cap \mathcal{E}(G^F,s)= \{ \chi \in \operatorname{Irr}(G^F) \mid (L, \lambda) \ll_e (G, \chi) \}$.
\end{enumerate}
\end{thm} 

\noindent
Throughout this work we will work under the following core assumption.
\begin{assumption} \label{assumption} $G$ is connected reductive, defined over $\mathbb{F}_q$ with Frobenius endomorphism $F:G \to G$, $\ell \nmid q$ is odd and good for $G$ and $e=e_\ell(q)$. If $G^F$ has a component of type $^3\!D_4$ then $\ell \geq 5$. Furthermore, $s \in G^{*F}$ will be a semisimple $\ell'$-element.
\end{assumption}
The following proposition shows how the notion of \textit{generalized $e$-Harish-Chandra theory holding in a Lusztig series} is related to the parametrisation of blocks by the $e$-cuspidal pairs. Let $\mathcal{E}(G^F,(L, \lambda)):= \{ \chi \in \operatorname{Irr}(G^F) \mid (L, \lambda) \leq_e (G, \chi) \}$ be the \textbf{$e$-Harish-Chandra series} associated to an $e$-cuspidal pair $(L, \lambda)$.

\begin{prop}{\cite[Proposition 2.2.2]{EnguehardJD}}\label{Prop:e-Harish-Chandra blockwise} Suppose Assumption \ref{assumption} holds. Then generalized $e$-Harish-Chandra theory holds in $\mathcal{E}(G^F,s)$ if and only if, for any $e$-cuspidal pair $(L, \lambda)$ of $G^F$ with $\lambda \in \mathcal{E}(L^F,s)$, we have
\begin{align*}
\mathcal{E}(G^F,(L, \lambda))=\{ \chi \in \operatorname{Irr}(G^F) \mid (L, \lambda) \ll_e (G, \chi) \}.
\end{align*}
\end{prop}

\begin{cor} \label{e-Harish-Chandra alternative} Suppose Assumption \ref{assumption} holds. Then generalized $e$-Harish-Chandra theory holds in $\mathcal{E}(G^F,s)$ if and only if 
\begin{align*}
\mathcal{E}(G^F,s)= \dot\bigcup_{(L,\lambda)/G^F} \mathcal{E}(G^F,(L, \lambda)),
\end{align*}
where $(L, \lambda)$ runs over the $G^F$-conjugacy classes of $e$-cuspidal pairs of $G$ with $s \in L^{*F}$ and $\lambda \in \mathcal{E}(L^F,s)$.
\end{cor}

\begin{proof}
By Theorem \ref{blocks good l} we have
\begin{align*}
\mathcal{E}(G^F,s) &= \dot \bigcup_{(L,\lambda)/G^F} \left( \operatorname{Irr}(b_{G^F}(L, \lambda)) \cap \mathcal{E}(G^F,s) \right) \\
				&= \dot \bigcup_{(L,\lambda)/G^F} \{ \chi \in \operatorname{Irr}(G^F) \mid (L, \lambda) \ll_e (G, \chi) \},
\end{align*}
where $(L, \lambda)$ runs over the $G^F$-conjugacy classes of $e$-cuspidal pairs of $G$ with $s \in L^{*F}$ and $\lambda \in \mathcal{E}(L^F,s)$. Since $\mathcal{E}(G^F,(L, \lambda))$ is contained in $\{ \chi \in \operatorname{Irr}(G^F) \mid (L, \lambda) \ll_e (G, \chi) \}$ by definition, the assertion follows from Proposition \ref{Prop:e-Harish-Chandra blockwise}.
\end{proof}

\begin{remark}
Note that, even though the statements about $e$-cuspidal pairs and $e$-Harish-Chandra theory in this section seem like they do not depend on $\ell$, the proofs of these statements heavily rely on $\ell$ satisfying the conditions in Assumption \ref{assumption} as can be seen in the proofs of the results cited in this section.
\end{remark}

The next tool we need in order to prove that generalized $e$-Harish-Chandra theory holds in the various Lusztig series is the \textit{Mackey formula}. Unfortunately it is not yet known if this formula holds in full generality, which forces us to define the following notion. We say that \textbf{the Mackey formula holds} for $G^F$, if for all $F$-stable Levi subgroups $N$ of $G$ the Mackey formula 
\begin{align*}
^*R_{L\subseteq P}^N \circ R_{M \subseteq Q}^N=\sum_x R_{L \cap ^x\!M \subseteq L \cap ^x\!Q}^L \circ ^*\!R_{L \cap ^x\!M \subseteq P \cap ^x\!M}^{^x\!M} \circ \operatorname{ad}x,
\end{align*}
where $x$ runs over a set of representatives of $L^F \setminus \mathcal{S}(L,M,N)^F/M^F$ with $\mathcal{S}(L,M,N)= \{x \in N \mid L \cap \text{}^x\!M \text{ contains a maximal torus of } N \}$, holds for every pair of parabolic subgroups $P$ and $Q$ of $N$ with $F$-stable Levi complements $L$ and $M$ respectively. \\

Since we focus on the case where $G$ is a simple, simply connected algebraic group of exceptional type, the only cases for which the Mackey formula is still open are  $^2\! E_6(2)$ and $E_8(2)$ (see \cite[Theorem]{Bonnafe-Michel} and \cite[Theorem 1.6]{Jay}). Recall that an immediate consequence of the Mackey formula is that Lusztig induction $R_{L \subseteq P}^G$ and restriction $^*\!R_{L \subseteq P}^G$ are independent of the chosen parabolic $P$ containing $L$. Because of this, we will omit the parabolic subgroups from the subscript of Lusztig induction and restriction from now on. \\

The following result by Enguehard may serve as a first indicator of how important the Mackey formula is in the context of this paper. The result yields the assertion of Theorem \ref{e-Harish-Chandra good prime} for the groups of type $F_4$ and the groups of type $E_8$ as long as we assume $q >2$.

\begin{thm}[{\cite[2.2.4 Proposition]{EnguehardJD}}] \label{e-Harish-Chandra connected center}
Suppose that Assumption \ref{assumption} holds. In addition suppose that the centre of $G$ is connected and that the Mackey formula holds for $G^F$. Then generalized $e$-Harish-Chandra theory holds in $\mathcal{E}(G^F,s)$.
\end{thm}

\noindent
By results of Kessar--Malle \cite{Malle-Kessar}, the assertion of Theorem \ref{e-Harish-Chandra good prime} also holds when $e=1$ or $e=2$ unless $G=E_6$ or $E_7$ and $s$ is semisimple, quasi-isolated of order 6. 
Thus, apart from these two exceptions (which we include in our treatment of $E_6$ and $E_7$ in Section 3), we can focus our attention on the situation where $G^F=E_6(q)$, $^2\!E_6(q)$, $E_7(q)$ or $E_8(2)$ and $e \geq 3$. The proof of Theorem \ref{e-Harish-Chandra good prime} is done case-by-case. However, since we need to tweak our argument slightly when $q=2$, we put that part of the proof at the end of Section 3. \\

For most $e \in \mathbb{N}$ the assertion of Theorem \ref{e-Harish-Chandra good prime} is obviously true. In order to distinguish these cases from the ones for which we need to put in work, we introduce the following notion.

\begin{definition}
For a semisimple element $s \in G^{*F}$ we define $\delta(G^F,s):=\{e \in \mathbb{N} \text{ } | \text{ } \exists \text{ a } \linebreak \text{proper }  e \text{-cuspidal pair } (L, \lambda) \text{ of } G^F \text{ with } \lambda \in \mathcal{E}(L^F,s) \}$. We say an integer $e$ is \textbf{relevant} for a semisimple element $s \in G^{*F}$ if it occurs in $\delta(G^F,s)$. Furthermore, if $e$ is relevant, then we call a proper $e$-cuspidal pair $(L, \lambda)$ of $G^F$ \textbf{relevant}.
\end{definition}
\noindent
The following easy conclusion justifies this terminology.

\begin{prop} \label{Proof wenn e zu gross ist}
Let $G$ be a connected reductive group defined over $\mathbb{F}_q$ with Frobenius endomorphism $F:G \to G$. Let $s \in G^{*F}$ be semisimple. If $e$ is not relevant for $s$, then generalized $e$-Harish-Chandra theory holds in $\mathcal{E}(G^F,s)$.
\end{prop}

\begin{proof} We start with an easy observation: If $\chi \in \mathcal{E}(G^F,s)$ and $(L, \lambda)$ is an $e$-split pair with $(L, \lambda) \ll_e (G, \chi)$, then $\lambda \in \mathcal{E}(L^F,s)$. Since $\ll_e$ is the transitive closure of $\leq_e$, there exists an $e$-split pair $(M, \zeta)$ such that $(L, \lambda) \leq_e (M, \zeta)$ and $(M, \zeta) \leq_e (G, \chi)$. By the definition of $\leq_e$ and \cite[Corollary 6]{Lusztig}, it follows that $\zeta \in \mathcal{E}(M^F,s)$ and then using \cite[Corollary 6]{Lusztig} again it follows that $\lambda \in \mathcal{E}(L^F,s)$. 

For every $\chi \in \mathcal{E}(G^F,s)$ there exists a unique $e$-cuspidal pair $(L, \lambda)$ (up to $G^F$-conjugacy) such that $(L, \lambda) \ll_e (G, \chi)$ by Theorem \ref{blocks good l}, and by our observation we know that $\lambda \in \mathcal{E}(L^F,s)$. Now, if $e$ is not relevant for $s$, then $(L, \lambda)=(G, \chi)$. Hence, every $\chi \in \mathcal{E}(G^F,s)$ is $e$-cuspidal. Since, clearly,
\begin{align*}
\mathcal{E}(G^F,s)= \dot\bigcup_{\chi \in \mathcal{E}(G^F,s)} \mathcal{E}(G^F, (G, \chi))
\end{align*}
the assertion follows by Corollary \ref{e-Harish-Chandra alternative}.
\end{proof}

Next we will show that we can determine the relevant integers by using only unipotent data. This follows mainly from the following result. In fact, we will use this result throughout Section 3 as well.

For every pair $(L,s)$ consisting of an $F$-stable Levi subgroup of $G$ and a semisimple element $s \in L^{*F}$ we fix a Jordan decomposition $J_s^L: \mathcal{E}(L^F,s) \to \mathcal{E}(C_{L^*}^\circ(s),1)$ as in \cite[Proposition 5.1]{Lusztigdisconnected}.

\begin{thm}[{\cite[Theorem 4.2.]{Cabanes-Enguehardtwisted}}] \label{e-cuspidal pairs}
Suppose that Assumption \ref{assumption} holds. Then an element $\chi \in \mathcal{E}(G^F,s)$ is $e$-cuspidal if and only if it satisfies the following conditions. 
\begin{enumerate}
\item $Z^\circ(C_{G^*}^\circ(s))_{\Phi_e}=Z^\circ(G^*)_{\Phi_e}$ and 
\item $J_s^G(\chi)$ is a $C_{G*}(s)^F$-orbit of $e$-cuspidal unipotent character of $C_{G^*}^\circ(s)^F$.
\end{enumerate}
\end{thm}
\noindent 

\begin{prop} \label{intersection levi centraliser} Let $G$ be a connected reductive group, $s \in G$ a semisimple element and $L \subseteq G$ a Levi subgroup of $G$ containing $s$. Then $C_L^\circ(s)=L \cap C_G^\circ(s)$ is a Levi subgroup of $C_G^\circ(s)$ and every Levi subgroup of $C_{G}^\circ(s)$ is of that form. Moreover,  $L \cap C_{G}^\circ(s) \subseteq C_{G}^\circ(s)$ is $e$-split if and only if $L \subseteq G$ is $e$-split. 
\end{prop}

\begin{proof}
As a semisimple element, $s$ lies in at least one maximal torus $S$ of $L$, which then is a maximal torus of $G$. Now, $Z(L)$ lies in every maximal torus of $L$. In particular, $Z(L)$ lies in $S$. In other words, we have $Z(L) \subseteq S \subseteq  C_G^\circ(s)$.
As $L=C_G(Z^\circ(L))$ (see \cite[1.21 Proposition]{Digne-Michel}), we have $L \cap C_G^\circ(s)=C_{C_G^\circ(s)}(Z^\circ(L))$. Since $Z^\circ(L)$ is a torus of $C_G^\circ(s)$, $C_{C_G^\circ(s)}(Z^\circ(L))$ is a Levi subgroup of $C_G^\circ(s)$, proving the first part.

Let $M$ be a Levi subgroup of $C_{G}^\circ(s)$.  Then $M\!=\!C_{C_G^\circ(s)}(Z^\circ(M))$. Now, $L=C_{G}(Z^\circ(M))$ is a Levi subgroup such that $M=L \cap C_{G}^\circ(s)$. The second part follows from Proposition \ref{e-split Levi}. 
\end{proof}

\begin{prop}\label{proper e-cuspidal pairs} 
Suppose that Assumption \ref{assumption} holds with $s \in G^{*F}$  a semisimple, quasi-isolated $\ell'$-element. Then $\delta(G^F,s)= \delta(C_{G^*}^\circ(s)^F,1)$.
\end{prop}

\begin{proof}
Let $(L, \lambda)$ be a proper $e$-cuspidal pair of $G$ with $\lambda \in \mathcal{E}(L^F,s)$. Let $L^*$ denote the dual of $L$ in $G^*$. To prove the assertion we show that Jordan decomposition yields a $C_{L^*}(s)^F$-orbit of proper unipotent $e$-cuspidal pairs.

By \cite[Proposition 1.4]{Cabanes-Enguehardunipotent2}, $L^*$ is $e$-split. Hence $L^*=C_{G^*}(Z^\circ(L^*)_{\Phi_e})$ by Proposition \ref{e-split Levi}. Since $s \in G^*$ is quasi-isolated, we know that $C_{G^*}^\circ(s) \not\subset L^*$. It follows that $C_{L^*}^\circ(s)=L^* \cap C_{G^*}^\circ(s) \subsetneq C_{G^*}^\circ(s)$ is a proper subgroup of $C_{G^*}^\circ(s)$. Furthermore, by Proposition \ref{intersection levi centraliser}, $C_{L^*}^\circ(s)$ is an $e$-split Levi subgroup of $C_{G^*}^\circ(s)$. Moreover, $\lambda$ corresponds to a $C_{L^*}(s)^F$-orbit of $e$-cuspidal unipotent characters of $C_{L^*}^\circ(s)^F$ by condition (ii) of Theorem \ref{e-cuspidal pairs}. Hence, $\delta(G^F,s) \subseteq \delta(C_{G^*}^\circ(s)^F,1)$. 

Conversely, let $(M, \chi)$ be a proper $e$-cuspidal pair of $C_{G^*}^\circ(s)$ with $\chi \in \mathcal{E}(C_{G^*}^\circ(s),1)$. By Proposition \ref{intersection levi centraliser}, there is a proper $e$-split Levi $L^* \subseteq G^*$ such that $L \cap C_{G^*}^\circ(s)=M$. If $\lambda$ is the character in $\mathcal{E}(L^F,s)$ mapped to $\chi$ by Jordan decomposition, then $(L, \lambda)$ is an $e$-cuspidal pair by Theorem \ref{e-cuspidal pairs}.
\end{proof}

This is helpful because it is very easy to determine the relevant $e$'s on the unipotent side of things with the help of CHEVIE \cite{Chevie}.

This idea of reducing computations to the unipotent side is a recurring theme in the representation theory of finite groups of Lie type and in this article.

\section{relevant $e$-cuspidal pairs}
\noindent
Let $G$ be a simple, simply connected algebraic group of exceptional type defined over $\mathbb{F}_q$ with Frobenius endomorphism $F: G \to G$ or let $G$ be simple, simply connected of type $D_4$ defined over $\mathbb{F}_q$ with Frobenius endomorphism $F:G \to G$ such that $G^F=\text{} ^3\!D_4(q)$.  \\

In this section we prove Theorem \ref{e-Harish-Chandra good prime}. Recall that an element $s$ of a connected reductive group $G$ is called \textbf{quasi-isolated} if $C_{G}(s)$ is not contained in any proper Levi subgroup $L \subsetneq G$. If even $C_{G}^\circ(s)$ is not contained in any proper Levi subgroup $L \subsetneq G$, then $s$ is called \textbf{isolated}. For the reader's convenience we restate the classification of the quasi-isolated elements here (see \cite[Proposition 4.3 and Table 3]{Bonnafe}).

\begin{prop}[Bonnafé] \label{Bonnafe}
Let $G$ be a simple, exceptional algebraic group of adjoint type. Then the conjugacy classes of semisimple, quasi-isolated elements $1 \neq s \in G$, their orders, the type of their centraliser $C_G(s)$, and the group of components $A(s):=C_G(s)/C_G^\circ(s)$ are as given in Table \ref{tabu:quasi-isolated elements}. 
\end{prop}
\noindent
The order of $s$ is denoted by $o(s)$.

\begin{longtable}{|c|c|l|c|l|} 
\caption{Quasi-isolated elements in exceptional groups} \\ \hline
\label{tabu:quasi-isolated elements}
$G$	&	$o(s)$		&	$C_G^\circ(s)$	&	$A(s)$	&	isolated? \\ \hline  \hline
$G_2$	&	2	&	$A_1 \times A_1$	&	1	&	yes \\
	&	3	&	$A_2$	&	1	&	 yes \\
$F_4$	&	2	&	$C_3 \times A_1, B_4$	&	1	&	yes\\
	&	3	&	$A_2 \times A_2$	&	1 	&	yes\\
	&	4	&	$A_3 \times A_1$	&	1 	&	yes\\
$E_6$	&	2	&	$A_5 \times A_1$	&	1 &	 yes\\
	&	3	&	$A_2 \times A_2 \times A_2$	&	3	&	yes\\
	&	3	&	$D_4$	&	3	&	no \\
	&	6	&	$A_1 \times A_1 \times A_1 \times A_1$	&	3 	&	no \\
$E_7$ &	2	&	$D_6 \times A_1$	&	1 &	yes\\
	&	2	&	$A_7$	&	2	&	yes \\
	&	2	&	$E_6$	&	2	&	no \\
	&	3	&	$A_5 \times A_2$	 &	1 	&	yes\\
	&	4	&	$A_3 \times A_3 \times A_1$	&	2	&	yes \\
	&	4	&	$D_4 \times A_1 \times A_1$	&	2	&	no	\\
	&	6	&	$A_2 \times A_2 \times A_2$	&	2 	&	no \\
$E_8$	&	2	&	$D_8, E_7 \times A_1$	&	1	&	yes \\
	&	3	&	$A_8, E_6 \times A_2$	&	1 	&	yes	\\
	&	4	&	$D_5 \times A_3, A_7 \times A_1$	&	1	&	yes \\
	&	5	&	$A_4 \times A_4$	 &	1	&	yes \\
	&	6	&	$A_5 \times A_2 \times A_1$	&	1	&	yes	 \\ \hline
\end{longtable}

In order to prove Theorem 1.1 we first determine all relevant $e$-cuspidal pairs for quasi-isolated semisimple elements of $G$. For a given $e$-cuspidal pair $(L, \lambda)$ of $G^F$ we then determine the constituents of $R_L^G(\lambda)$. This can be done ad-hoc for each individual $e$-cuspidal pair. With the help of new results from the most recent book \cite{Geck-Malle} of M. Geck and G. Malle, determining the constituents of $R_L^G(\lambda)$ is immediate in all cases for the groups of type $G_2$, $F_4$ and $E_8$: \\

The following statement is contained in \cite[Theorem
 4.7.5]{Geck-Malle}.

\begin{thm}{\cite[Theorem 4.7.5]{Geck-Malle}} \label{Jordan decomposition commutes connected centre}
Let $G$ be simple with connected center, $F:G \to G$ a Steinberg endomorphism, and assume that the Mackey formula holds for $G^F$. Let $s \in G^{*F}$ be a semisimple element. Then for all $F$-stable Levi subgroups $L^* \subseteq M^* \subseteq G^*$ satisfying $s \in L^*$, with duals $L \subseteq M \subseteq G$, the diagram
\begin{center}
\begin{tikzcd}
\mathbb{Z} \mathcal{E}(M^F,s)\arrow[r, "J_s^M"]&  \mathbb{Z}\mathcal{E}(C_{M^*}(s)^F,1)  \\ \mathbb{Z} \mathcal{E}(L^F,s) \arrow[u, "R_L^M"] \arrow[r, "J_s^L"]& \mathbb{Z} \mathcal{E}(C_{L^*}(s)^F,1) \arrow[u, "R_{C_{L^*}(s)}^{C_{G^*}(s)}"]
\end{tikzcd}
\end{center}
commutes, except possibly when $G=M$ is of type $E_8$, $C_{G^*}(s)$ is of type $E_6 \times A_2$ or $E_7 \times A_1$ and $C_{L^*}(s)$ has a factor of type $E_6$ or $E_7$.
\end{thm}

If $G$ and $s \in G^{*F}$ are as in Theorem \ref{Jordan decomposition commutes connected centre} and $(L,\lambda)$ is an $e$-cuspidal pair of $G$, then
\[J_s^G(R_L^G(\lambda)=R_{C_{L*}(s)}^{C_{G*}(s)}(J_s^L(\lambda))\]
Since $G$ has connected center, $J_s^G$ is a bijection and therefore $R_L^G(\lambda)$ is uniquely determined by the right-hand side, which can be easily computed with CHEVIE \cite{Chevie}. \\

\subsection{The tables}
The tables are structured as follows. The first column numbers the different (rational) Lusztig series parametrized by the ($G^{*F}$-conjugacy classes of) semisimple quasi-isolated elements of $G^{*F}$ for a given $e$. Let $s$ be such a quasi-isolated element. The second column gives $C_{G^*}(s)^F$.  The fourth column gives the $e$-split Levi $L$ of $G$ such that $(L, \lambda)$ is an $e$-cuspidal pair of $G^F$, where $\lambda$ is one of the characters (if there are multiple ones) in the sixth column. The fifth column contains $C_{L^*}(s)^F$. The second to last column contains the order of the relative Weyl groups $W_{G^F}(L, \lambda):= N_{G^F}(L, \lambda)/L^F$. These entries are especially important since $|W_{G^F}(L, \lambda)|=||R_L^G(\lambda)||^2$, which is needed to determine $R_L^G(\lambda)
$ in the nontrivial cases (i.e. when $\lambda$ is not uniform and Theorem \ref{Jordan decomposition commutes connected centre} can not be applied). The last column gives the number of irreducible Brauer characters in the block corresponding to the $e$-cuspidal pair $(L, \lambda)$. \\
 
Given a number $i$, the information in row $i$  is computed as follows. Let $s \in G^{*F}$ be the quasi-isolated element associated with $i$. First we determine the unipotent $e$-cuspidal pairs of $C_{G^*}^\circ(s)^F$. By early work of Lusztig, information on unipotent characters is readily available and the unipotent $e$-cuspidal pairs can easily be computed using CHEVIE \cite{Chevie}. Next we compute the $e$-cuspidal pairs of $G^F$. The proof of Proposition \ref{proper e-cuspidal pairs} shows that every $e$-cuspidal pair $(L, \lambda)$ of $G^F$ is a lift of a unipotent $e$-cuspidal pair $(L_s, \lambda_s)$ of $C_{G^*}^\circ(s)^F$, such that $L_s=C_{L^*}^\circ(s)$ and  $\lambda_s \in J_s^L(\lambda)$. In particular, if $C_{L^*}(s)^F \neq C_{L^*}^\circ(s)^F$, there may be multiple unipotent $e$-cuspidal pairs that lift to $(L, \lambda)$, namely all pairs of the form $(C_{L^*}^\circ(s)^F, \mu)$ with $\mu \in J_s^L(\lambda)$. Hence, we need to study the $C_{L^*}(s)^F$-orbits of the unipotent $e$-cuspidal characters of $C_{L^*}^\circ(s)^F$. It turns out that the only non-trivial orbits occur when $C_{L^*}(s)$ is disconnected and of type $A_1 \times A_1$ in the case where $G$ is of type $E_7$. The computation of the relative Weyl  group is fairly straighforward. Apart from a few exceptions we have 
\begin{align}
W_{G^F}(L, \lambda)=N_{G^F}(L, \mathcal{E}(L^F,s))/L^F,
\end{align}
which can be computed using \cite[(2.2.5.2)]{EnguehardJD}.
By \cite[Proposition 1.9 (i)]{Cabanes-Enguehardtwisted} we know that $W_{G^F}(L, \lambda) \subseteq N_{G ^F}(L^F, \bar{\mathcal{E}}(L^F,s))/L^F$, where $\bar{\mathcal{E}}(L^F,s)$ denotes the geometric Lusztig series (see \cite[Definition 13.16]{Digne-Michel}). When $C_{L^*}(s)^F=C_{L^*}^\circ(s)^F$, geometric and rational Lusztig series coincide, i.e. we have  $W_{G^F}(L, \lambda) \subseteq N_{G ^F}(L^F, \mathcal{E}(L^F,s))/L^F$. If, in addition, the degree of $\lambda$ is unique in $\mathcal{E}(L^F,s)$, (1) follows from the fact that degrees are invariant under conjugation. When $C_{L^*}(s)^F\ \neq C_{L^*}^\circ(s)^F$, it turns out that two $e$-cuspidal characters in $\mathcal{E}(L^F,s)$ have the same degree if and only if they are send to the same orbit via Jordan decomposition. In particular, (1) then follows from \cite[Proposition 2.4.2 (a)]{EnguehardJD}.

Exceptions to (1) only occur when $C_{L^*}(s)^F=C_{L^*}^\circ(s)^F$ involves  $A_1(q)^2$. Here, we sometimes have characters of the form $\lambda \otimes \phi_{11} \otimes \phi_{2}$ which are conjugate to $\lambda \otimes \phi_2 \otimes \phi_{11}$. However, in all theses cases it can be explicitly checked whether or not they are actually conjugate. \\

In addition to the Tables \ref{table:Quasi-isolated blocks F4 good l}, \ref{Table:Kessar-Malle E_6 good l}, \ref{Table: Quasi-isolated blocks of E_7(q) good l}, \ref{Table: Quasi-isolated blocks of E_8(q) good l} and \ref{Table: Quasi-isolated blocks of G_2(q) and 3D4(q) good l}, Tables \ref{table:Lusztig Induction E6 good l} and \ref{table:Lusztig Induction E7 good l} contain the decomposition of $R_L^G(\lambda)$ into its irreducible constituents for every $e$-cuspidal pair $(L, \lambda)$ for which $R_L^G( \lambda)$ is not uniform (defined later) or cannot be determined using Theorem \ref{Jordan decomposition commutes connected centre}. 
The constituents of $R_L^G(\lambda)$ are parametrized via Jordan decomposition. 
Since the semisimple element will always be clear from the context, we omit it from the parametrization and denote every irreducible constituent by the corresponding unipotent character. Except for the unipotent characters of classical groups (where we use the common notation using partitions and symbols), we use the notation of CHEVIE \cite{Chevie}. \\

It should be noted that we do not include tables for every relevant integer $e$. The missing tables, however, are Ennola duals of the ones in this section and they can be obtained fairly easily. This follows from the fact that Ennola duality of finite groups of Lie type interacts nicely with Lusztig induction and restriction (see \cite{BMM} and especially \cite[3.3 Theorem]{BMM}). The Ennola dual cases are $e=1 \leftrightarrow e=2$, $ e=3 \leftrightarrow e=6$, $e=5 \leftrightarrow e=10$, $e=7 \leftrightarrow e=14$, $e=9 \leftrightarrow e=18$ and $e=15 \leftrightarrow e=30$. For reference, we repeat the following remark from the introduction.

\begin{remark}\label{e=1 and e=2} The $e$-cuspidal pairs of $G^F$ for $e=1 \leftrightarrow e=2$ were already determined by Kessar and Malle in \cite{Malle-Kessar} except for the pairs associated to quasi-isolated elements of order 6 when $G^F=E_6(q)$ or $E_7(q)$. 
\end{remark}

\subsection{$e$-cuspidal pairs of $F_4$} \label{Appendix F4}

Let $G$ be simple, simply connected of type $F_4$ defined over $\mathbb{F}_q$ with Frobenius endomorphism $F: G \to G$. In this case, $e$ is relevant for some quasi-isolated semisimple $1 \neq s \in G^{*F}$ if and only if  $e \in \{1,2,3,4,6\}$. By Remark \ref{e=1 and e=2} and Ennola duality, it remains to determine the $e$-cuspidal pairs for $e=3$ and $e=4$. 

\begin{thm}\label{e-cuspidal pairs F4 good l}
Let $e=e_\ell(q) \in \{3,4\}$. Let $1 \neq s \in G^{*F}$ be semisimple, quasi-isolated and such that $e$ is relevant for $s$. Then the $e$-cuspidal pairs $(L, \lambda)$ of $G$ with $\lambda \in \mathcal{E}(L^F,s)$ (up to $G^F$-conjugacy), and the order of their relative Weyl groups $W=W_{G^F}(L, \lambda)$ are as indicated in Table \ref{table:Quasi-isolated blocks F4 good l}. In particular, generalized $e$-Harish-Chandra theory holds in $\mathcal{E}(G^F,s)$ for every quasi-isolated semisimple element $1 \neq s \in G^{*F}$.
\end{thm}

\begin{longtable}{|c|l|c|l|l|l|c|c|} 
\caption{Quasi-isolated blocks in $F_4(q)$} \\ \hline
\label{table:Quasi-isolated blocks F4 good l} 
No. & $C_{G^*}(s)^F$  & $e$ & $L^F$ & $C_{L^*}(s)^F$ & $\lambda$  & $|W|$ & $l(B)$ \\ \hline \hline
1 & $A_2(q)A_2(q)$ & 3 & $\Phi_3^2$ & $\Phi_3^2$ & 1 & 9 & 9 \\ \hline
2 & $B_4(q)$ & 3 & $\Phi_3.\tilde{A_2}(q)$ & $\Phi_1\Phi_3.\tilde{A_1}(q)$ & $\phi_{11}, \phi_2$ & 6 & 6 \\ 
3 & & 3 & $G^F$ & $C_{G^*}(s)^F$ & 13 chars. & 1 & 1\\ \hline
4 & $C_3(q)A_1(q)$ & 3 & $\Phi_3.A_2(q)$ & $\Phi_1\Phi_3. A_1(q)$ & $\phi_{11},\phi_2$ & 6 & 6\\
5 & & 3 & $G^F$ & $C_{G^*}(s)^F$ &  12 chars. & 1 & 1\\ \hline
6 & $A_3(q)\tilde{A_1}(q)$ & 3 & $\Phi_3.A_2(q)$ & $\Phi_1 \Phi_3. \tilde{A_1}(q)$ & $\phi_{11},\phi_2$ &3 & 3 \\ 
7 & & 3 & $G^F$ & $C_{G^*}(s)^F$ & 4  chars. & 1 & 1 \\ \hline \hline
1 & $B_4(q)$ & 4 & $\Phi_4^2$ & $\Phi_4^2$ & 1 & 32 & 14\\ 
2 & & 4 & $\Phi_4.B_2(q)$ & $\Phi_4.B_2(q)$ & $(12,0),(01,2)$ & 4 & 4 \\
3 & & 4 & $G^F$ & $C_{G^*}(s)^F$ & $(13,1),(013,13)$, & 1 & 1 \\
& & & & & $(014,12)$ & & \\ \hline
4 & $C_3(q)A_1(q)$ & 4 & $\Phi_4.B_2(q)$ & $\Phi_4.A_1(q)A_1(q)$ & $\phi_{11} \otimes \phi_{11}, \phi_{11} \otimes \phi_2$, & 4 & 4\\
& & & & &  $\phi_2 \otimes \phi_{11}, \phi_2 \otimes \phi_2$ & & \\
5 & & 4 & $G^F$ & $C_{G^*}(s)^F$ & 8 chars. & 1 & 1 \\ \hline
6 & $A_3(q)A_1(q)$ & 4 & $\Phi_4.B_2(q)$ & $\Phi_2\Phi_4. \tilde{A_1}(q)$ & $\phi_{11}, \phi_2$ & 4 & 4\\ 
7 & & 4 & $G^F$ & $C_{G^*}(s)^F$ & $\phi_{22} \otimes \phi_{11}, \phi_{22} \otimes \phi_2$ & 1 & 1\\ \hline
8 & $^2\!A_3(q)\tilde{A_1}(q)$ & 4 & $\Phi_4.B_2(q)$ & $\Phi_1 \Phi_4. \tilde{A_1}(q)$ & $\phi_{11}, \phi_2$ & 4 & 4\\
9 & & 4 & $G^F$ & $C_{G^*}(s)^F$ & $\phi_{22} \otimes \phi_{11}, \phi_{22} \otimes \phi_2$ & 1 & 1\\ \hline
\end{longtable}

\begin{proof}
By Theorem \ref{Jordan decomposition commutes connected centre} we can easily determine $R_L^G(\lambda)$ for each $e$-cuspidal pair in the table. It is easily checked that $\mathcal{E}(G^F,s)$ is the union of the corresponding $e$-Harish-Chandra series as in Corollary \ref{e-Harish-Chandra alternative}, which proves the assertion.
\end{proof}

Let $\pi_{uni}$ denote the \textbf
{uniform projection}, that is the projection from the space of class functions onto the subspace of uniform functions (see \cite[12.11 Definition]{Digne-Michel}). The image of a class function under $\pi_{uni}$ can be explicitly computed using  \cite[12.12 Proposition]{Digne-Michel}. \\

For further reference we showcase the main method used in Sections 3.3 and 3.4 to determine the constituents of the $R_L^G(\lambda)$'s. We do so here because for $G$ of type $E_6$ and $E_7$ the calculations are too lengthy to get a grasp of the underlying idea.

\begin{example} \label{showcase}
Let $e=4$ and let $(L, \lambda)=(B_2, (12,0))$ be a $4$-cuspidal pair of $F_4(q)$. In this case  $\pi_{uni}(R_L^G(\lambda))=\frac{1}{4} [ (1234,012) - (123,02) + (023,12) - (0124,123) + (0123,124) - (23,0)  + (14,0) - (02,3) + (01,4) + (023,-) - (014,-) + (0123,2)] - \frac{3}{4} [(03,2) - (012,23) - (04,1) - (01234,12)] \in \frac{1}{4} \mathbb{Z} \mathcal{E}(G^F,s)$. Since $R_{L}^G(\lambda)$ is a generalized character, there exists an element $\gamma \in \mathbb{Q}\mathcal{E}(G^F,s)$ which is orthogonal to the space of uniform class functions of $G^F$, such that 
$ R_L^G(\lambda)=\pi_{uni}(R_{L}^G(\lambda))+ \gamma \in \mathbb{Z}\mathcal{E}(G^F,s)$. A basis for the subspace of $\mathbb{Q} \mathcal{E}(G^F,s)$ orthogonal to the space of uniform class functions is given by
\begingroup
\allowdisplaybreaks
\begin{align*}
\varphi_1&=\frac{1}{4} \left((1234,012)-(0124,123)+(0123,124)-(01234,12) \right), \\
\varphi_2&= \frac{1}{4} \left((123,02)- (023,12) + (012,23) - (0123,2) \right), \\
\varphi_3&= \frac{1}{4} \left( (124,01) - (014,12) + (012,14) - (0124,1) \right), \\
\varphi_4&= \frac{1}{4} \left( (23,0)-(03,2)+(02,3)-(023,-) \right), \\
\varphi_5&= \frac{1}{4} \left( (14,0) - (04,1) + (01,4) - (014,-) \right).
\end{align*} 
\endgroup
This was determined as follows. First, we determine $\mathbb{Q} \mathcal{E}(G^F,s)_{uni}$, the subspace of uniform class functions in $\mathbb{Q}\mathcal{E}(G^F,s)$. This is easy since a basis of that space is given by the different Deligne--Lusztig characters corresponding to $s$ whose deomposition into irreducible characters is known by work of Lusztig \cite{Orangebook}. Computing the orthogonal complement of that subspace then yields $\{\varphi_1, \dots , \varphi_5\}$.
By the Mackey formula we know that $\|R_L^G(\lambda)\|^2=|W_{G^F}(L,\lambda)|=4$ and since $\|R_L^G(\lambda)\|^2= \|\pi_{uni}(R_{L}^G(\lambda))\|^2+ \| \gamma \|^2$, it follows that $\gamma= - \varphi_1 + \varphi_2+\varphi_4-\varphi_5$. Hence, $R_L^G(\lambda)=-(03,2)+(012,23)+(04,1)+ (01234,12)$.
\end{example}

\subsection{$e$-cuspidal pairs of $E_6$}\label{Appendix E6} 

Let $G$ be a simple, simply connected 
of type $E_6$ defined over $\mathbb{F}_q$ with Frobenius endomorphism $F:G \to G$. Then $G^F= E_{6,sc}(q)$ or $^2\!E_{6,sc}(q)$. We start with $G^F=E_{6,sc}(q)$. Here, $e$ is relevant for some quasi-isolated $1 \neq s \in G^{*F}$ if and only if $e \in \{1,2,3,4,5,6\}$. Since the center of $G$ is disconnected, the situation is more complicated than for $F_4$. \\

\noindent
In the tables in this section and Section 3.4 we write $C_{G^*}(s)^F=C_{G^*}^\circ(s)^F.(C_{G^*}(s)^F/C_{G^*}^\circ(s)^F)$ when the centraliser is disconnected. A star in the first column next to the number of the line indicates that the quotient $C_{G^*}(s)^F/C_{G^*}^\circ(s)^F$ acts non-trivially on the unipotent characters of $C_{G^*}^\circ(s)^F$. 
To demonstrate the adjustments, we take line 4 of Table \ref{Table:Kessar-Malle E_6 good l} for $e=3$ as an example: 

First, the star indicates that the $F$-stable points of the component group act non-trivially on the 14 unipotent characters of $C_{G^*}^\circ(s)^F=\Phi_1^2.D_4(q)$. It can be shown that there are two orbits of order 3, and 8 trivial orbits. Thus, by Jordan decomposition, $|\mathcal{E}(G^F,s)|=26$.
Now, $C_{L^*}(s)^F/C_{L^*}^\circ(s)^F$ obviously acts trivially on the one unipotent character (which is the trivial character) of the torus $C_{L^*}^\circ(s)^F=\Phi_1^4 \Phi_3$. Hence the induction of that character to $C_{L^*}(s)^F$ yields 3 irreducible constituents. We denote them by $1^{(1)},1^{(2)}$ and $1^{(3)}$.  

In general, if $C_{G^*}(s)^F/C_{G^*}^\circ(s)^F$ acts trivially on a given unipotent character of $C_{G^*}^\circ(s)^F$, the induction of that character always yields 3 irreducible characters of $C_{G^*}(s)^F$. In Table \ref{Table:Kessar-Malle E_6 good l}, we indicate this by adding a superscript from 1 to 3 to that unipotent character. 

\begin{thm} \label{e-cuspidal pairs E6 good l}
Let $e=e_\ell(q) \in \{1,2,3,4,5,6\}$ and let $1 \neq s \in G^{*F}$ be semisimple, quasi-isolated and such that $e$ is relevant for $s$. If $e \in \{1,2\}$ we only consider semisimple, quasi-isolated elements of order 6.
Then the $e$-cuspidal pairs $(L, \lambda)$ of $G$ with $\lambda \in \mathcal{E}(L^F,s)$ (up to $G^F$-conjugacy), and the order of their relative Weyl groups $W=W_{G^F}(L, \lambda)$ are as indicated in Table \ref{Table:Kessar-Malle E_6 good l}. In particular, generalized $e$-Harish-Chandra theory holds in $\mathcal{E}(G^F,s)$ for every quasi-isolated semisimple element $1 \neq s \in G^{*F}$.
\end{thm}
\text{ } \newline
\begin{longtable}{|c|l|c|l|l|l|c|c|} 
\caption{Quasi-isolated blocks in $E_6(q)$} \\ \hline
\label{Table:Kessar-Malle E_6 good l}
No. & $C_{G^*}(s)^F$  & $e$ & $L^F$ & $C_{L^*}(s)^F$ & $\lambda$ & $|W|$ & $l(B)$ \\ \hline \hline
$1^*$ & $\Phi_1^2.A_1(q)^4.3$ & 1 & $\Phi_1^6$ & $\Phi_1^6$ & $1$ & $16$ & 16 
\\ \hline \hline
$1^*$ & $\Phi_1^2. A_1(q)^4.3$ & 2 & $\Phi_1^2 \Phi_2^4$ & $\Phi_1^2 \Phi_2^4$ & 1 & 16 & 16 \\ \hline
$1^*$ & $A_2(q)^3.3$ & 3 & $\Phi_3^3$ & $\Phi_3^3$ & 1 & 81 & 81 \\ \hline
2 & $A_2(q^3).3$ & 3 & $\Phi_3^2.A_2(q)$ & $\Phi_1^2 \Phi_3^2$ & 1 & 9  & 9  \\ \hline
3 & $^2\!A_2(q)A_2(q^2)$ & 3 &  $\Phi_3.^3\!D_4(q)$  & $\Phi_3 \Phi_6.^2\!A_2(q)$ & $\phi_{111}, \phi_{21}, \phi_3$ & 3 & 3\\ \hline
$4^*$ & $\Phi_1^2.D_4(q).3$ & 3 & $\Phi_3.A_2(q)^2$ & $\Phi_1^4 \Phi_3.3$ & $1^{(1)},1^{(2)},1^{(3)}$ & 6 & 6\\ 
$5^*$ & & 3 & $G^F$ & $C_{G^*}(s)^F$ & 8 chars.  & 1 & 1\\ \hline
6 & $\Phi_3.^3\!D_4(q).3$ & 3 & $\Phi_3^3$ & $\Phi_3^3$ & 1 & 72 & 21\\
7 & & 3 & $\Phi_3.^3\!D_4(q)$ & $\Phi_3.^3\!D_4(q)$ & $^3\!D_4[-1]$ & 3 & 3\\ \hline
8 & $\Phi_1 \Phi_2. ^2\!D_4(q)$ & 3 & $\Phi_3.A_2(q)^2$ & $\Phi_1^2 \Phi_2^2  \Phi_3$ & $1$  & 6 & 6\\ 
9 & & 3 & $G^F$ & $C_{G^*}(s)^F$ & 4 chars. & 1 & 1 \\ \hline
10 & $A_5(q)A_1(q)$ & 3 & $\Phi_3^2. A_2(q)$ & $\Phi_1. \Phi_3. A_1(q)$ & $\phi_{11}, \phi_2$ & 18 & 9 \\
11 & & 3 & $G^F$ & $C_{G^*}(s)^F$ & 4 chars. & 1 & 1 \\ \hline \hline
1 & $A_5(q)A_1(q)$ & 4 & $\Phi_1 \Phi_4. ^2\!A_3(q)$ & $\Phi_1 \Phi_2 \Phi_4. A_1(q)^2$ & 4 chars. & 4 & 4 \\
2 & & 4 & $G^F$ & $C_{G^*}(s)^F$ & 6 chars. & 1 & 1 \\ \hline
$3^*$ & $\Phi_1^2.D_4(q).3$ & 4 & $\Phi_1^2 \Phi_4^2$ & $\Phi_1^2 \Phi_4^2$ & 1 & 48 & 30 \\
$4^*$ & & 4 & $G^F$ & $C_{G^*}(s)^F$ & 12 chars. & 1 & 1\\ \hline
5 & $\Phi_1 \Phi_2.^2\!D_4(q)$ & 4 & $\Phi_1 \Phi_4.^2\!A_3(q)$ & $\Phi_1 \Phi_2 \Phi_4.A_1(q)^2$ & $\phi_{11} \otimes \phi_{11}$, & 4 & 4\\ 
& & & & & $\phi_{2} \otimes \phi_2$ & 4 & 4\\
& & & & & $\phi_{11} \otimes \phi_2$ & 2 & 2\\ \hline \hline
1 & $A_5(q)A_1(q)$ & 5 & $\Phi_1 \Phi_5.A_1(q)$ & $\Phi_1 \Phi_5.A_1(q)$ & $\phi_{11}, \phi_{2}$ & 5 & 5\\ 
2 & & 5 & $G^F$ & $C_{G^*}(s)^F$ & 12 chars. & 1 & 1\\ \hline  \hline
1 & $A_5(q)A_1(q)$ &  6 & $\Phi_3 \Phi_6. ^2\!A_2(q)$ & $\Phi_2 \Phi_3 \Phi_6.A_1(q)$ & $\phi_{11}, \phi_{2}$ & 6 & 6\\ 
2 & & 6 & $G^F$ & $C_{G^*}(s)^F$ & 10 chars. & 1 & 1\\ \hline
$3^*$ & $\Phi_1^2.D_4(q).3$ & 6 & $\Phi_6.A_2(q^2)$ & $\Phi_1^2 \Phi_2^2 \Phi_6.3$ & $1^{(1)},1^{(2)},1^{(3)}$ & 6 & 6\\ 
$4^*$ & & 6 & $G^F$ & $C_{G^*}(s)^F$ & 8 chars. & 1 & 1\\ \hline
5 & $\Phi_3.^3\!D_4(q).3$ & 6 & $\Phi_3 \Phi_6^2$ & $\Phi_3 \Phi_6^2$ & 1 & 72 & 21\\ 
6 & & 6 & $G^F$ & $C_{G^*}(s)^F$ & $\phi_{2,1}$ & 1 & 1\\ \hline 
7 & $\Phi_1 \Phi_2.^2\!D_4(q)$ & 6 & $\Phi_6.A_2(q^2)$ & $\Phi_1^2  \Phi_2^2 \Phi_6$ & 1 & 6 & 6 \\ 
8 & & 6 & $G^F$ & $C_{G^*}(s)^F$ & 4 chars. & 1 & 1\\ \hline
\end{longtable}

\normalsize

\begin{proof}
For $q=2$ the assertion follows from Proposition \ref{Exceptions, Mackey not proved}. Suppose $q>2$. Except for the pairs given in Table \ref{table:Lusztig Induction E6 good l}, $\lambda$ is uniform, so $R_L^G(\lambda)$ can be determined using the formula for the uniform projection. For the $3$-cuspidal pairs $(\Phi_3.A_2(q)^2, 1^{(i)})$ ($i=1, \dots, 3$) and the $6$-cuspidal pairs $(\Phi_6.A_2(q^2),1^{(i)})$ ($i=1, \dots, 3$), we are not able to determine $R_L^G( \lambda)$ (see Remark \ref{Problems with R_L^G}). However the methods used in Example \ref{showcase} give enough information to prove that an $e$-Harish Chandra theory holds in the Lusztig series related to the $e$-cuspidal pairs above. For the $3$-cuspidal pair $(\Phi_3.^3\!D_4(q), ^3\!D_4[-1])$ we use a slightly different argument. Let $s \in G^{*F}$ be semisimple and quasi-isolated with $C_{G^*}(s)^F= \Phi_3.^3\!D_4(q).3$. By Table \ref{Table:Kessar-Malle E_6 good l} and Theorem \ref{e-cuspidal pairs}, $\mathcal{E}(G^F,s)$ decomposes into two blocks, namely $b_{G^F}(\Phi_3^3,1)$, which contains $\mathcal{E}(G^F,(\Phi_3^3,1))$ and $b_{G^F}(\Phi_3.^3\!D_4(q), ^3\!D_4[-1])$ which contains $\mathcal{E}(G^F, (\Phi_3.^3\!D_4(q),^3\!D_4[-1]))$. Since any two different blocks, seen as subsets of $\operatorname{Irr}(G^F) \cup \operatorname{IBr}(G^F)$, are disjoint, we have 
\[
\mathcal{E}(G^F, (\Phi_3.^3\!D_4(q),^3\!D_4[-1]) \subseteq \mathcal{E}(G^F,s) \setminus \mathcal{E}(G^F,(\Phi_3^3,1)) \] and the latter is equal to $\{^3\!D_4[-1]^{(0)}, ^3\!D_4[-1]^{(1)}, ^3\!D_4[-1]^{(2)}\}$.
Since $R_{\Phi_3.^3\!D_4(q)}^G(^3\!D_4[-1])$ has norm 3, it follows that  $R_{\Phi_3.^3\!D_4(q)}^G(^3\!D_4[-1])=\text{}^3\!D_4[-1]^{(0)}+^3\!D_4[-1]^{(1)}+^3\!D_4[-1]^{(2)}$.
Hence, an $e$-Harish-Chandra theory holds in $\mathcal{E}(G^F,s)$.
\end{proof}

\begin{remark} \label{Problems with R_L^G}
The reason we are not able to determine $R_L^G(1^{(i)})$ in the cases numbered $4^*$ ($e=3$) and $3^*$ ($e=6$) in Table \ref{table:Lusztig Induction E6 good l} is part of a general problem of parametrising the characters in a Lusztig series in a canonical way when the corresponding centralizer is disconnected.  In the cases above, every constituent of $R_L^G(1^{(i)})$ is an element of an orbit of order 3. However, we are not able to determine which element of this orbit is the right constituent. We only know that it has to be one of the three. This is indicated by adding a superscript $(i)$ to the constituents.
\end{remark}

\begin{longtable}{|c|c|l|p{12,2cm}|}
\caption{Decomposition of the non-uniform $R_L^G(\lambda)$} \\ \hline
\label{table:Lusztig Induction E6 good l}
No. & $e$ & $\lambda$ & $\pm R_L^G(\lambda)$ \\ \hline \hline
$4^*$ & 3 & $1^{(i)}$ & $(013,123)^{(i)}+(0123,1234)^{(i)}+ (02,13)^{(i)}+(01,23)^{(i)}+(1,3)^{(i)}+(0,4)^{(i)}$ \\ \hline
7 & 3 & $^3\!D_4[-1]$ & $^3\!D_4[-1]^{(0)}+^3\!D_4[-1]^{(1)}+^3\!D_4[-1]^{(2)}$\\ \hline
$3^*$ & 6 & $1^{(i)}$ &  $(013,123)^{(i)}+(0123,1234)^{(i)}+(12,03)^{(i)}+(1,3)^{(i)}+(0,4)^{(i)}+(0123,-)^{(i)}$ \\ \hline
\end{longtable}

The analogue of Table \ref{Table:Kessar-Malle E_6 good l} for $^2\!E_6(q)$ can be obtained as follows. 
The $e=3$ part of the table for $^2\!E_6(q)$ is the Ennola dual of the $e=6$ part of Table \ref{Table:Kessar-Malle E_6 good l} and vice-versa. The $e=10$ part is the Ennola dual of the $e=5$ part and the $e=4$ part is the Ennola dual of the $e=4$ part of Table \ref{Table:Kessar-Malle E_6 good l}. Similarly, the analogue of Table \ref{table:Lusztig Induction E6 good l} for $^2\!E_6(q)$ can be obtained via Ennola duality. Thus, the assertion of Theorem \ref{e-cuspidal pairs E6 good l} holds for $^2\!E_{6,sc}$ as well.

\subsection{$e$-cuspidal pairs of $E_7$}
Let $G$ be a simple, simply connected group of type $E_7$ defined over $\mathbb{F}_q$ with Frobenius endomorphism $F:G \to G$. In this case, $e$ is relevant for some quasi-isolated semisimple $1 \neq s \in G^{*F}$ if and only if $e \in \{1,2,3,4,5,6,7,9,12,14,18\}$. By Remark \ref{e=1 and e=2} and Ennola duality, it remains to determine the $e$-cuspidal pairs for $e \in \{1,3,4,5,7,9,12\}$. Since the center of $G$ is disconnected, we encounter the same issues as in Section \ref{Appendix E6}.

\begin{thm} \label{e-cuspidal pairs $E_7$ good l}
Let $e=e_\ell(q) \in \{1,3,4,5,7,9,12\}$ and let $1 \neq s \in G^{*F}$ be semisimple, quasi-isolated and such that $e$ is relevant for $s$. If $e=1$ we only consider semisimple, quasi-isolated elements of order 6. Then the $e$-cuspidal pairs $(L, \lambda)$ of $G$ with $\lambda \in \mathcal{E}(L^F,s)$ (up to $G^F$-conjugacy), and the order of their relative Weyl groups $W=W_{G^F}(L, \lambda)$ are as indicated in Table \ref{Table: Quasi-isolated blocks of E_7(q) good l}. In particular, generalized $e$-Harish-Chandra theory holds in $\mathcal{E}(G^F,s)$ for every quasi-isolated semisimple element $1 \neq s \in G^{*F}$.
\end{thm}

\tiny

\begin{longtable}{|c|l|c|l|l|l|c|c|} 
\caption{Quasi-isolated blocks of $E_7(q)$} \\ \hline
\label{Table: Quasi-isolated blocks of E_7(q) good l}
No. & $C_{G^*}(s)^F$  & $e$ & $L^F$ & $C_{L^*}(s)^F$ & $\lambda$  & $|W|$ & $l(B)$ \\ \hline \hline
$1^*$ & $\Phi_1.A_2(q)^3.2$ & 1 & $\Phi_7$ & $\Phi_7$ & 1 & 270 & 27 \\ \hline
$2^*$ & $\Phi_2.^2\!A_2(q)^3.2$ & 1 & $\Phi_1^3.A_1(q)^4$ & $\Phi_1^3 \Phi_2^4$ & 1 & 10 & 10 \\ 
$3^*$ & & 1 & $\Phi_1^2.D_4(q)A_1(q)$ & $\Phi_1^2 \Phi_2^3.^2\!A_2(q)$ & $\phi_{21}$ & 5 & 5 \\
$4^*$ & & 1 & $\Phi_1^2.D_4(q)A_1(q)$ & $\Phi_1^2 \Phi_2^3.^2\!A_2(q)$ & $\phi_{21}$ & 4 & 4 \\
$5^*$ & & 1 & $\Phi_1.D_6(q)$ & $\Phi_1 \Phi_2^2.^2\!A_2(q)^2$ & $\phi_{21} \otimes \phi_{21}$ & 4 & 4 \\
$6^*$ & & 1 & $\Phi_1.D_6(q)$ & $\Phi_1 \Phi_2^2.^2\!A_2(q)^2$ & $\phi_{21} \otimes \phi_{21}$ & 2 & 2 \\
$7^*$ & & 1 & $G^F$ & $C_{G^*}(s)^F$ & $(\phi_{21}\otimes \phi_{21} \otimes \phi_{21})^{(1,2)}$ & 1 & 1 \\ \hline \hline
1 & $A_7(q).2$ & 3 & $\Phi_1 \Phi_3^2.A_2(q)$ & $\Phi_1^2 \Phi_3^2.A_1(q)$ & $\phi_{11}, \phi_2$ & 36 & 18 \\
2 & & 3 & $\Phi_3.A_5(q)$ & $\Phi_1 \Phi_3.A_4(q)$ & $\phi_{311}$ & 6 & 6 \\
3 & & 3 & $G^F$ & $C_{G^*}(s)^F$ & $\phi_{4211}^{(1,2)}$  & 1 & 1\\ \hline
4 & $^2\!A_7(q).2$ & 3 & $\Phi_3.A_2(q)A_1(q^3)$ & $\Phi_1 \Phi_2 \Phi_3 \Phi_6. A_1(q).2$ & $\phi_{11}^{(1,2)}$, $\phi_2^{(1,2)}$ & 6 & 6\\
5 & & 3 & $G^F$ & $C_{G^*}(s)^F$ & 16 chars. & 1 & 1\\ \hline
6 & $\Phi_1.E_6(q).2$ & 3 & $\Phi_1 \Phi_3^3$ & $\Phi_1 \Phi_3^3$ & 1 & 1296 & 48\\ 
7 & & 3 & $\Phi_1 \Phi_3.^3\!D_4(q)$ & $\Phi_1 \Phi_3.^3\!D_4(q)$ & $^3\!D_4[-1]$ & 6 & 6\\ 
8 & & 3 & $G^F$ & $C_{G^*}(s)^F$ &  6 chars. & 1 & 1\\ \hline
9 & $\Phi_2.^2\!E_6(q).2$ & 3 &  $\Phi_3^2.A_1(q^3)$ & $\Phi_2 \Phi_3^2 \Phi_6.2$ & $1^{(1,2)}$ & 72 & 21 \\
10 & & 3 & $G^F$ & $C_{G^*}(s)^F$ & 18 chars. & 1 & 1\\ \hline
$11^*$ & $A_3(q)^2A_1(q).2$ & 3 & $\Phi_1 \Phi_3^2. A_2(q)$ & $\Phi_1^2 \Phi_3^2. A_1(q)$ & $\phi_{11}, \phi_2$ & 9 & 9 \\
$12^*$ & & 3 & $\Phi_1 \Phi_3.A_5(q)$ & $\Phi_1 \Phi_3. A_1(q)A_3(q)$ & 4 chars. & 3 & 3 \\
$13^*$ & & 3 & $G^F$ & $C_{G^*}(s)^F$ & 10 chars. & 1 & 1\\ \hline
14 & $A_3(q^2)A_1(q).2$ & 3 & $\Phi_1 \Phi_3.^3\!D_4(q)$  & $\Phi_1 \Phi_2 \Phi_3 \Phi_6.A_1(q)$ & $\phi_{11}, \phi_2$ & 6 & 6\\ 
15 & & 3 & $G^F$ & $C_{G^*}(s)^F$ & 8 chars. & 1 & 1 \\ \hline
16 & $A_3(q^2)A_1(q).2$ & 3 & $\Phi_3.A_2(q)A_1(q^3)$ & $\Phi_1 \Phi_2 \Phi_3 \Phi_6.A_1(q).2$ & $\phi_{11}^{(1,2)}, \phi_2^{(1,2)}$ & 3 & 3 \\ 
17 & & 3 & $G^F$ & $C_{G^*}(s)^F$ & 8 chars. & 1 & 1 \\ \hline
$18^*$ & $\Phi_1.A_2(q)^3.2$ & 3 & $\Phi_1 \Phi_3^3$ & $\Phi_1 \Phi_3^3$ & 1 & 54 & 54 \\ \hline
19 & $A_5(q)A_2(q)$ & 3 & $\Phi_1 \Phi_3^3$ & $\Phi_1 \Phi_3^3$ & 1 & 54 & 27 \\
20 & & 3 & $\Phi_3.A_5(q)$ & $\Phi_3.A_5(q)$ & $\Phi_{2^21^2}, \Phi_{42}$ & 3 & 3 \\ \hline
21 & $^2\!A_5(q) ^2\!A_2(q)$ & 3 & $\Phi_1 \Phi_3.^3\!D_4(q)$ & $\Phi_1 \Phi_3 \Phi_6. ^2\!A_2(q)$ & $\phi_{111}, \phi_{21}, \phi_3$ & 6 & 6 \\ 
22 & & 3 & $G^F$ & $C_{G^*}(s)^F$ & 15 chars. & 1 & 1 \\ \hline
$23^*$ & $\Phi_1.D_4(q) A_1(q)^2.2$ & 3 &  $\Phi_3. A_5(q)$ & $\Phi_1^3 \Phi_3. A_1(q)^2.2$ & 5 chars.  & 6 & 6 \\
$24^*$ & & 3 & $G^F$ & $C_{G^*}(s)^F$ & 40 chars. & 1 & 1\\ \hline
25 & $\Phi_2.^2\!D_4(q)A_1(q^2).2$ & 3 & $\Phi_3.A_5(q)$ & $\Phi_1 \Phi_2^2 \Phi_3.A_1(q^2).2$ & $\phi_{11}^{(1,2)}, \phi_2^{(1,2)}$ & 6 & 6\\ 
26 & & 3 & $G^F$ & $C_{G^*}(s)^F$ & 16 chars. & 1 & 1 \\ \hline 
$27^*$ & $\Phi_2.D_4(q)A_1(q)^2.2$ & 3 &  $\Phi_3.A_5(q)$ & $\Phi_1^2 \Phi_2 \Phi_3 A_1(q)^2.2$ & 5 chars. & 6 & 6\\
$28^*$ & & 3 & $G^F$ & $C_{G^*}(s)^F$ & 40 chars. & 1 & 1\\ \hline
29 & $\Phi_1. ^2\!D_4(q)A_1(q^2).2$ & 3 & $\Phi_3.A_5(q)$ & $\Phi_1^2 \Phi_2 \Phi_3.A_1(q^2).2$ & $\phi_{11}^{(1,2)}, \phi_2^{(1,2)}$ & 6 & 6 \\ 
30 & & 3 & $G^F$ & $C_{G^*}(s)^F$ & 16 chars. & 1 & 1 \\ \hline 
31 & $D_6(q)A_1(q)$ & 3 & $\Phi_1 \Phi_3^2. A_2(q)$ & $\Phi_1^2 \Phi_3^2$ & $\phi_{11}, \phi_2$ & 36 & 18 \\ 
32 & & 3 & $\Phi_3.A_5(q)$ & $\Phi_1 \Phi_3. A_3(q) A_1(q)$ & 4 chars. & 6 & 6 \\ 
33 & & 3 & $G^F$ & $C_{G^*}(s)^F$ & 24 chars. & 1 & 1 \\ \hline \hline
1 & $A_7(q).2$ & 4 & $\Phi_4^2. A_1(q)^3$ & $\Phi_1 \Phi_2^2 \Phi_4^2.2$ & $1^{(1,2)}$ & 32 & 14 \\ 
2 & & 4 & $\Phi_4.^2\!D_4(q)A_1(q)$ & $\Phi_2^2 \Phi_4.A_3(q).2$ & $\phi_{22}^{(1,2)}$ & 4 & 4 \\
3 & & 4 & $G^F$ & $C_{G^*}(s)^F$ & 8 chars. & 1& 1 \\ \hline
4 & $^2\!A_7(q).2$ & 4 & $\Phi_4^2.A_1(q)^3$ & $\Phi_1^2 \Phi_2 \Phi_4^2.2$ & $1^{(1,2)}$ & 32 & 14 \\
5 & & 4 & $\Phi_4.^2\!D_4(q)A_1(q)$ & $\Phi_1 \Phi_2 \Phi_4.^2\!A_3(q).2$ & $\phi_{22}^{(1,2)}$ & 4 & 4 \\ 
6 & & 4 & $G^F$ & $C_{G^*}(s)^F$ & 8 chars. & 1 & 1\\ \hline
7 & $\Phi_1.E_6(q).2$ & 4 & $\Phi_4^2.A_1(q)^3$ & $\Phi_1^3 \Phi_4^2.2$ & $1^{(1,2)}$ & 96 & 16\\
8 & & 4 & $\Phi_4. ^2\!D_4(q)A_1(q)$ & $\Phi_1^2 \Phi_4.^2\!A_3(q).2$ & $\phi_{22}^{(1,2)}$ & 4 & 4 \\
9 & & 4 & $G^F$ & $C_{G^*}(s)^F$ & 20 chars. & 1 & 1\\ \hline
10 & $\Phi_2.^2\!E_6(q).2$ & 4 &  $\Phi_4^2.A_1(q)^3$ & $\Phi_2^3 \Phi_4^2.2$ & $1^{(1,2)}$ & 96 & 16\\ 
11 & & 4 & $\Phi_4.^2\!D_4(q)A_1(q)$ & $\Phi_2^2 \Phi_4.A_3(q).2$ & $\phi_{22}^{(1,2)}$ & 4 & 4 \\
12 & & 4 & $G^F$ & $C_{G^*}(s)^F$ & 20 chars. & 1 & 1 \\ \hline
$13^*$ & $A_3(q)^2A_1(q).2$ & 4 &  $\Phi_4^2.A_1(q)^3$ & $\Phi_2^2 \Phi_4^2.A_1(q)$ & $\phi_{11}, \phi_2$ & 14 & 14 \\ 
$14^*$ & & 4 & $\Phi_4.^2\!D_4(q)A_1(q)$ & $\Phi_2 \Phi_4.A_3(q)A_1(q)$ & $\phi_{22} \otimes \phi_{11}$, & 4 & 4\\
 & & & & & $\phi_{22} \otimes \phi_2$ & 4 & 4\\
$15^*$ & & 4 & $G^F$ & $C_{G^*}(s)^F$ & 4 chars. & 1 & 1 \\ \hline
16  & $A_3(q^2)A_1(q).2$ & 4 & $\Phi_4^2.A_1(q)^3$ & $\Phi_1 \Phi_2 \Phi_4^2.A_1(q)$ & $\phi_{11}, \phi_2$ & 16 & 10 \\ \hline
$17^*$ & $^2\!A_3(q)^2A_1(q).2$ & 4 &  $\Phi_4^2.A_1(q)^3$ & $\Phi_1^2 \Phi_4^2.A_1(q)$ & $\phi_{11}, \phi_2$ & 14 & 14 \\
$18^*$ & & 4 &  $\Phi_4.^2\!D_4(q)A_1(q)$ & $\Phi_1 \Phi_4.^2\!A_3(q)A_1(q)$ & $\phi_{22} \otimes \phi_{11}$, & 4 & 4 \\
& & & & & $\phi_{22} \otimes \phi_2$ & 4 & 4\\
$19^*$ & & 4 & $G^F$ & $C_{G^*}(s)^F$ & 4 chars. & 1 & 1\\ \hline
24 & $A_5(q) A_2(q)$ & 4 & $\Phi_4.^2\!D_4(q)A_1(q)$ & $\Phi_1 \Phi_2 \Phi_4. A_2(q)A_1(q)$ & 6 chars. & 4 & 4 \\ 
25 & & 4 & $G^F$ & $C_{G^*}(s)^F$ & 9 chars. & 1 & 1 \\ \hline
26 & $^2\!A_5(q) ^2\!A_2(q)$ & 4 &  $\Phi_4. ^2\!D_4(q) A_1(q)$ & $\Phi_1 \Phi_2 \Phi_4.^2\!A_2(q)$ & 6  chars. & 4 & 4\\
27 & & 4 & $G^F$ & $C_{G^*}(s)^F$ & 9 chars. & 1 & 1 \\ \hline
$28^*$ & $\Phi_1.D_4(q)A_1(q)^2.2$ & 4 &  $\Phi_4^2.A_1(q)^3$ & $\Phi_1 \Phi_4^2.A_1(q)^2$ & $\phi_{11} \otimes \phi_{11}$,  & 32 & 20 \\
& & & & & $\phi_{11} \otimes \phi_{2}$, & 16 & 10\\
& & & & &  $\phi_{2} \otimes \phi_2$   & 32 & 20\\
$29^*$ & & 4 & $G^F$ & $C_{G^*}(s)^F$ & 20 chars. & 1 & 1 \\ \hline
30 & $\Phi_2.^2\!D_4(q)A_1(q^2).2$ & 4 &  $\Phi_4^2.A_1(q)^3$ & $\Phi_2 \Phi_4^2.A_1(q)^2$ & $\phi_{11} \otimes \phi_{11}$ & 16 & 16\\ 
& & 4 & & & $\phi_2 \otimes \phi_2$ & 16  & 16 \\
& & 4 & & & $\phi_{11} \otimes \phi_2$ & 8 & 8\\ \hline
$31^*$ & $\Phi_2.D_4(q)A_1(q)^2.2$ & 4 &  $\Phi_4^2.A_1(q)^3$ & $\Phi_2 \Phi_4^2.A_1(q)^2$ & $\phi_{11} \otimes \phi_{11}$,  & 32 & 20\\
& & & & & $\phi_{11} \otimes \phi_{2}$, & 16 & 10 \\
& & & & &  $\phi_{2} \otimes \phi_2$   & 32 & 20\\
$32^*$ & & 4 & $G^F$ & $C_{G^*}(s)^F$ & 20 chars. & 1 & 1\\ \hline
33 & $\Phi_1.^2\!D_4(q)A_1(q^2).2$ & 4 & $\Phi_4^2.A_1(q)^3$ & $\Phi_1 \Phi_4^2.A_1(q)^2$ & $\phi_{11} \otimes \phi_{11}$ & 16 & 16\\ 
& & 4 & & & $\phi_2 \otimes \phi_2$ & 16 & 16\\
& & 4 & & & $\phi_{11} \otimes \phi_2$ & 8 & 8\\\hline \hline
1 & $A_7(q).2$ & 5 & $\Phi_1 \Phi_5.A_2(q)$ & $\Phi_1 \Phi_5. A_2(q)$ & $\phi_{111}, \phi_{21}, \phi_3$ & 10 & 10 \\
2 & & 5 & $G^F$ & $C_{G^*}(s)^F$ & 14 chars. & 1 & 1 \\ \hline
3 & $\Phi_1.E_6(q).2$ & 5 & $\Phi_1 \Phi_5. A_2(q)$ & $\Phi_1^2 \Phi_5. A_1(q)$ & $\phi_{11}, \phi_2$ & 10 & 10\\
4 & & 5 & $G^F$ & $C_{G^*}(s)^F$ & 40 chars. & 1 & 1 \\ \hline 
5 & $A_5(q)A_2(q)$ & 5 & $\Phi_1\Phi_5.A_2(q) $ & $\Phi_1 \Phi_5. A_2(q)$ & $\phi_{111}, \phi_{21}, \phi_{3}$ & 5 & 5 \\ 
6 & & 5 & $G^F$ & $ C_{G^*}(s)^F$ & 18 chars. & 1 & 1 \\ \hline 
7 & $D_6(q)A_1(q)$ & 5 & $\Phi_1 \Phi_5.A_2(q)$ & $ \Phi_1^2 \Phi_5. A_1(q)$ & $\phi_{11}, \phi_2$ & 10 & 10 \\ 
8 & & 5 & $G^F$ & $C_{G^*}(s)^F$ & 64 chars. & 1 & 1 \\ \hline \hline
1 & $A_7(q).2$ & 7 & $\Phi_1 \Phi_7$ & $\Phi_1 \Phi_7$ & 1 & 14 & 14\\
2 & & 7 & $G^F$ & $C_{G^*}(s)^F$ & 30 chars. & 1 & 1\\ \hline \hline
1 & $\Phi_1.E_6(q).2$ & 9 & $\Phi_1 \Phi_9$ & $\Phi_1 \Phi_9$ & 1 & 18 & 18\\
2 & & 9 & $G^F$ & $C_{G^*}(s)^F$ & 42 chars. & 1 & 1 \\ \hline \hline
1 & $\Phi_1.E_6(q).2$ & 12 & $\Phi_{12}.A_1(q^3)$ & $\Phi_1 \Phi_3 \Phi_{12}$ & 1 & 24 & 24 \\
2 & & 12 & $G^F$ & $C_{G^*}(s)^F$ & 36 chars. & 1 & 1\\ \hline
1 & $\Phi_2.^2\!E_6(q).2$ & 12 & $\Phi_{12}.A_1(q^3)$ & $\Phi_2 \Phi_6 \Phi_{12}$ & 1 & 24 & 24\\ 
2 & & 12 & $G^F$ & $C_{G^*}(s)^F$ & 36 chars. & 1 & 1 \\  \hline
\end{longtable}

\normalsize

\begin{proof}
Similar to the proof of Theorem \ref{e-cuspidal pairs E6 good l}.
\end{proof}

\begin{longtable}{|c|p{0,2cm}|p{2,2cm}|p{12cm}|}
\caption{Decomposition of the non-uniform $R_L^G(\lambda)$} \\ \hline
\label{table:Lusztig Induction E7 good l}
No. & $e$ &  $\lambda$ & $\pm R_L^G(\lambda)$ \\ \hline \hline
4 & 3 & $\phi_{11}^{(i)}$ & $\phi_{1^8}^{(i)}+\phi_{2^311}^{(i)}+ \phi_{3221}^{(i)}+ \phi_{422}^{(i)} + \phi_{62}^{(i)} + \phi_{71}^{(i)}$ \\ \cline{2-4}
& 3 & $\phi_2^{(i)}$ & $\phi_{21^6}^{(i)}+\phi_{221^4}^{(i)}+\phi_{3311}^{(i)}+\phi_{431}^{(i)}+\phi_{53}^{(i)}+\phi_8^{(i)}$ \\ \hline
7 & 3 & $^3\!D_4[-1]$ & $D_4\!:\!3^{(0)}+D_4\!:\!3^{(1)}+D_4\!:\!111^{(0)}+D_4\!:\!111^{(1)} - D_4\!:\!21^{(0)} - D_4\!:\!21^{(1)}$\\ \hline
9 & 3 & $1^{(i)}$ & $\phi_{1,0}^{(i)}+ \phi_{1,24}^{(i)}-\phi_{2,4}^{''(i)}-\phi_{2,16}^{'(i)}+
\phi_{1,12}^{''(i)}+\phi_{1,12}^{'(i)}-2\phi_{4,1}^{(i)}-
2\phi_{4,13}+2\phi_{8,3}^{'(i)}+2\phi_{8,9}^{''(i)}-
\phi_{2,4}^{'(i)}-\phi_{2,16}^{''(i)}+\phi_{4,8}^{(i)}+
3 \text{ }^2\!E_6[1]^{(i)}-2 \phi_{4,7}^{''(i)}-2 \phi_{4,7}^{'(i)}+2 \phi_{8,3}^{''(i)}+2 \phi_{8,9}^{'(i)}-2 \phi_{16,5}^{(i)} -3 \text{ } ^2\!E_6[\theta]^{(i)}- 3 \text{ } ^2\!E_6[\theta^2]^{(i)}$\\ \hline
$23^*, 27^*$ & 3 & $(\phi_2 \otimes \phi_2)^{(i)}$ & $((013,123) \otimes \phi_2 \otimes \phi_2)^{(i)}+ ((0123,1234) \otimes \phi_2 \otimes \phi_2)^{(i)}+$ \\
& & & $((02,13) \otimes \phi_2 \otimes \phi_2)^{(i)}+((01,23) \otimes \phi_2 \otimes \phi_2)^{(i)}+((1,3) \otimes \phi_2 \otimes \phi_2)^{(i)} + ((0,4) \otimes \phi_2 \otimes \phi_2)^{(i)}$ \\ \cline{2-4}
& 3 & $\phi_{11} \otimes \phi_2$ & $((013,123) \otimes \phi_{11} \otimes \phi_2)+ ((0123,1234) \otimes \phi_{11} \otimes \phi_2)+$ \\
& & & $((02,13) \otimes \phi_{11} \otimes \phi_2)+((01,23) \otimes \phi_{11} \otimes \phi_2)+$ \\ 
& & & $((1,3) \otimes \phi_{11} \otimes \phi_2) + ((0,4) \otimes \phi_{11} \otimes \phi_2)$ \\ \cline{2-4}
& 3 & $(\phi_{11} \otimes \phi_{11})^{(i)}$ & $((013,123) \otimes \phi_{11} \otimes \phi_{11})^{(i)}+ ((0123,1234) \otimes \phi_{11} \otimes \phi_{11})^{(i)}+$ \\
& & & $((02,13) \otimes \phi_{11} \otimes \phi_{11})^{(i)}+((01,23) \otimes \phi_{11} \otimes \phi_{11})^{(i)}+$ \\
& & & $((1,3) \otimes \phi_{11} \otimes \phi_{11})^{(i)} + ((0,4) \otimes \phi_{11} \otimes \phi_{11})^{(i)}$ \\ \hline
25, 29 & 3 & $\phi_{11}^{(i)}$ & $((123,0) \otimes \phi_{11})^{(i)} + ((01234,123) \otimes \phi_2)^{(i)} + ((13,) \otimes \phi_{11})^{(i)} + ((0123,13) \otimes \phi_{11})^{(i)} + ((04,) \otimes \phi_{11})^{(i)} + ((012,3) \otimes \phi_{11})^{(i)}$ \\ \cline{2-4}
& 3 & $\phi_2^{(i)}$ & $((123,0) \otimes \phi_2)^{(i)} + ((01234,123) \otimes \phi_2)^{(i)} + ((13,) \otimes \phi_2)^{(i)} +$ \\
& & & $((0123,13) \otimes \phi_2)^{(i)} + ((04,) \otimes \phi_2)^{(i)} + ((012,3) \otimes \phi_2)^{(i)}$ \\ \hline \hline
1 & 4 & $1^{(i)}$ & 
$\phi_{1^8}^{(i)}-\phi_{21^6}^{(i)}+2\phi_{2^4}^{(i)}+\phi_{31^5}^{(i)}-2\phi_{3221}^{(i)}+2\phi_{332}^{(i)}-\phi_{41^4}^{(i)}+2\phi_{4211}^{(i)}-2\phi_{431}^{(i)}+2 \phi_{44}^{(i)}-\phi_{51^3}^{(i)}+\phi_{611}^{(i)}-\phi_{71}^{(i)}+\phi_{8}^{(i)}$ \\ \hline
2 & 4 & $\phi_{22}^{(i)}$ & $- \phi_{221^4}^{(i)}+\phi_{2^311}^{(i)}-\phi_{53}^{(i)}+\phi_{62}^{(i)}$ \\ \hline
4 & 4 & $1^{(i)}$ & $\phi_{1^8}^{(i)}+ \phi_{21^6}^{(i)}+2\phi_{2^4}^{(i)}-\phi_{31^5}^{(i)}+2\phi_{3221}^{(i)}
-2\phi_{332}^{(i)}-\phi_{41^4}^{(i)}-2\phi_{4211}^{(i)}
+2\phi_{44}^{(i)}-\phi_{51^3}^{(i)}-\phi_{611}^{(i)}+\phi_{71}^{(i)}+\phi_8^{(i)}$ \\ \hline
5 & 4 & $\phi_{22}^{(i)}$ & $-\phi_{221^4}^{(i)}-\phi_{2^311}^{(i)}+\phi_{53}^{(i)}+\phi_{62}^{(i)}$ \\  \hline
7 & 4 & $1^{(i)}$ & $\phi_{1,0}^{(i)}+\phi_{1,36}^{(i)}+
2\phi_{10,9}^{(i)}+2\phi_{6,1}^{(i)}+2\phi_{6,25}^{(i)}
+\phi_{15,5}^{(i)}+\phi_{15,17}^{(i)}-3\phi_{15,4}^{(i)}
-3\phi_{15,16}^{(i)}+4\phi_{80,7}^{(i)}+
2\phi_{90,8}^{(i)}-3\phi_{81,6}^{(i)}-3\phi_{81,10}^{(i)}
-2 \text{ }D_4\!:\!3-2 \text{ }D_4\!:\!111-4 \text{ }D_4\!:\!21$ \\ \hline
8 & 4 & $\phi_{22}^{(i)}$ & $\phi_{20,2}^{(i)}-\phi_{20,20}^{(i)}-\phi_{60,5}^{(i)}
+\phi_{60,11}^{(i)}$\\ \hline 
10 & 4 & $1^{(i)}$ & $\phi_{1,0}^{(i)}+\phi_{1,24}^{(i)}-
2\phi_{6,6}^{'(i)}-\phi_{9,2}^{(i)}
-\phi_{9,10}^{(i)}+3\phi_{1,12}^{''(i)} 
+3\phi_{1,12}^{'(i)}+2\phi_{8,3}^{'(i)}
+2\phi_{8,9}^{''(i)}-2\phi_{2,4}^{'(i)}
-2\phi_{2,16}^{''(i)} -4 \text{ } ^2\!E_6[1]^{(i)} - 2 \phi_{6,6}^{''(i)}-3\phi_{9,6}^{''(i)}
-3\phi_{9,6}^{'(i)} + 4 \phi_{16,5}^{(i)}$\\ \hline
11 & 4 &  $\phi_{22}^{(i)}$ & $\phi_{4,1}^{(i)}-\phi_{4,13}^{(i)}
-\phi_{4,7}^{''(i)} + \phi_{4,7}^{'(i)}$\\ \hline
\end{longtable}

\subsection{$e$-cuspidal pairs of $E_8$}
Let $G$ be a simple, simply connected of type $E_8$ defined over $\mathbb{F}_q$ with Frobenius endomorphism $F: G \to G$. Here, $e$ is relevant for some quasi-isolated $1 \neq s \in G^{*F}$ if and only if $e \in \{1,2,3,4,5,6,7,9,10,12,14,$ $18,20\}$. By Remark \ref{e=1 and e=2} and Ennola duality, it remains to determine the $e$-cuspidal pairs for $e \in \{3,4,5,7,9,12,20\} $. 

\begin{thm} \label{e-cuspidal pairs $E_8$ good l}
Let $e=e_\ell(q) \in \{3,4,5,7,9,12,20\}$ and let $1 \neq s \in G^{*F}$ be semisimple, quasi-isolated and such that $e$ is relevant for $s$. Then the $e$-cuspidal pairs $(L, \lambda)$ of $G$ with $\lambda \in \mathcal{E}(L^F,s)$ (up to $G^F$-conjugacy), and the order of their relative Weyl groups $W=W_{G^F}(L, \lambda)$ are as indicated in Table \ref{Table: Quasi-isolated blocks of E_8(q) good l}. In particular, generalized $e$-Harish-Chandra theory holds in $\mathcal{E}(G^F,s)$ for every quasi-isolated semisimple element $1 \neq s \in G^{*F}$.
\end{thm}
\small

\begin{longtable}{|c|l|c|l|l|l|c|c|}
\caption{Quasi-isolated blocks of $E_8(q)$} \\ \hline
\label{Table: Quasi-isolated blocks of E_8(q) good l}
No. & $C_{G^*}(s)^F$  & $e$ & $L^F$ & $C_{L^*}(s)^F$ & $\lambda$  & $|W|$ & $l(B)$ \\ \hline \hline

1 & $E_7(q)A_1(q)$ & 3 & $\Phi_3^3. A_2(q)$ & $\Phi_1\Phi_3^3.A_1(q)$ & $\phi_1, \phi_{22}$ & 1296 & 48 \\
2 & & 3 &  $\Phi_3.^3\!D_4(q)A_2(q)$ & $\Phi_1\Phi_3. ^3\!D_4(q)A_1(q)$ & $^3\!D_4[-1] \otimes \phi_{11}$ & 6 & 6 \\
& & 3 & & & $^3\!D_4[-1] \otimes \phi_2$ & 6 & 6\\
3 & & 3 & $\Phi_3. E_6(q)$ & $\Phi_3. A_5(q)A_1(q)$ & 4 chars. & 6 & 6 \\
4 & & 3 & $G^F$ & $C_{G^*}(s)^F$ & 20 chars. & 1 & 1\\ \hline

5 & $E_6(q)A_2(q)$ & 3 & $\Phi_3^4$ & $\Phi_3^4$ & 1 & 1944 & 72 \\ 
6 & & 3 & $\Phi_3^2. ^3\!D_4(q)$ & $\Phi_3^2. ^3\!D_4(q)$ & $^3\!D_4[-1]$ & 9 & 9\\
7 & & 3 & $\Phi_3.E_6(q)$ & $\Phi_3.E_6(q)$ & $\phi_{81,6}, \phi_{81,10},$ & 3 & 3\\ 
 & & 3 &  &  & $\phi_{90,8}$ & 3 & 3\\ \hline
8 & $^2\!E_6(q)A_2(q)$ & 3 & $\Phi_3^2.^3\!D_4(q)$ & $\Phi_3^2 \Phi_6.^2\!A_2(q)$ & $\phi_{111}, \phi_{21}, \phi_3$ & 72 & 21 \\
9 & & 3 & $G^F$ & $C_{G^*}(s)^F$ & 27 chars. & 1 & 1 \\ \hline  
10 & $D_5(q)A_3(q)$ & 3 & $\Phi_3. A_2(q)^2$ & $\Phi_1^2 \Phi_3^2 A_1(q)^2$ & $\phi_{11} \otimes \phi_{11},$ & 18 & 18\\
& & 3 & & & $\phi_2 \otimes \phi_2$ & 18 & 18\\
11 & & 3 & & & $\phi_{11} \otimes \phi_2$ & 9 & 9 \\
12 & & 3 & $\Phi_3.E_6(q)$ & $\Phi_1 \Phi_3.A_3(q) A_1(q)^2$ & 4 chars. & 6 & 6\\ 
13 & & 3 & & & 4 chars. & 3 & 3 \\
14 & & 3 & & $\Phi_1\Phi_3.D_5(q)$ & 5 chars. & 3 & 3 \\
15 & & 3 & $G^F$ & $C_{G^*}(s)^F$ & 4 chars. & 1 & 1 \\ \hline
16 & $^2\!D_5(q)^2\!A_2(q)$ & 3 & $\Phi_3. E_6(q)$ & $\Phi_1 \Phi_3. ^2\!A_3(q)A_1(q^2)$ & 10 chars. & 6 & 6 \\ 
17 & & 3 & $G^F$ & $C_{G^*}(s)^F$ & 40 chars. & 1 & 1 \\ \hline
18 & $A_4(q)^2$ & 3 &  $\Phi_3^2.A_2(q)^2$ & $\Phi_1^2 \Phi_3^2.A_1(q)^2$ & 4 chars. & 9 & 9\\
19 & & 3 & $\Phi_3.E_6(q)$ & $\Phi_1 \Phi_3.A_4(q)A_1(q)$ & 4 chars. & 3 & 3 \\ 
20 & & 3 & $G^F$ & $C_{G^*}(s)^F$ & $\phi_{311} \otimes \phi_{311}$ & 1 & 1\\ \hline
21 & $A_5(q)A_2(q)A_1(q)$ & 3 & $\Phi_3^3.A_2(q)$ & $\Phi_1\Phi_3^3.A_1(q)$ & $\phi_{11}, \phi_2$ & 54 & 27\\
22 & & 3 & $\Phi_3.E_6(q)$ & $\Phi_3.A_5(q)A_1(q)$ & 4 chars. & 3 & 3 \\ \hline
23 & $^2\!A_5(q)^2\!A_2(q)A_1(q)$ & 3 & $\Phi_3.^3\!D_4(q)A_2(q)$ & $\Phi_1 \Phi_3 \Phi_6.^2\!A_2(q)A_1(q)
$ & 6 chars. & 6 & 6 \\
24 & & 3 & $G^F$ & $C_{G^*}(s)^F$ & 30 chars. & 1 & 1 \\ \hline
25 & $A_7(q)A_1(q)$ & 3 & $\Phi_3^2.A_2(q)^2$ & $\Phi_1^2\Phi_3^2.A_1(q)^2$ & 4 chars. & 18 & 9 \\
26 & & 3 & $\Phi_3.E_6(q)$ & $\Phi_1\Phi_3. A_4(q)A_1(q)$ & $\phi_{11} \otimes \phi_{311}$, &  3 & 3\\
 & & 3 & & & $\phi_2 \otimes \phi_{311}$ & 3 & 3\\ 
27 & & 3 & $G^F$ & $C_{G^*}(s)^F$ & $\phi_{4211} \otimes \phi_{11}$, & 1 & 1 \\
& & 3 & & & $\phi_{4211} \otimes \phi_2$ & 1 & 1 \\ \hline
28 & $^2\!A_7(q)A_1(q)$ & 3 & $\Phi_3.^3\!D_4(q)A_2(q)$ & $\Phi_1 \Phi_2 \Phi_3 \Phi_6.A_1(q)^2 $ & 4 chars. & 6 & 6 \\ 
29 & & 3 & $G^F$ & $C_{G^*}(s)^F$ & 20 chars. & 1 & 1 \\ \hline
30 & $A_8(q)$ & 3 & $\Phi_3^3.A_2(q)$ & $\Phi_1^2\Phi_3^3$ & 1 & 162 & 22\\
31 & & 3 & $\Phi_3.E_6(q)$ & $\Phi_1\Phi_3.A_5(q)$ & $\phi_{42}, \phi_{2211}$ & 3 & 3 \\ 
32 & & 3 & $G^F$ & $C_{G^*}(s)^F$ & $\phi_{32211}, \phi_{531}$ & 1 & 1\\ \hline
33 & $^2\!A_8(q)$ & 3 & $\Phi_3.^3\!D_4(q)A_2(q)$ & $\Phi_1 \Phi_2 \Phi_3 \Phi_6.^2\!A_2(q)$ & $\phi_{111}, \phi_{21}, \phi_3 $ & 6 & 6 \\ 
34 & & 3 & $G^F$ & $C_{G^*}(s)^F$ & 12 chars. & 1 & 1 \\ \hline
35 & $D_8(q)$ & 3 & $\Phi_3^2.A_2(q)^2$ & $\Phi_1^2 \Phi_3^2.A_1(q)^2$ & $\phi_{11} \otimes \phi_{11},$ & 72 & 27\\
& & 3 & & & $\phi_2 \otimes \phi_2$, & 72 & 27\\
36 & & 3 & & & $\phi_{11} \otimes \phi_2$ & 36 & 18\\
37 & & 3 & $\Phi_3.E_6(q)$ & $\Phi_1 \Phi_3.D_5(q)$ & 5 chars. & 6 & 6 \\ 
38 & & 3 & $G^F$ & $C_{G^*}(s)^F$ & 18  chars. & 1 & 1 \\ \hline \hline
1 & $E_7(q)A_1(q)$ & 4 & $\Phi_4^2.D_4(q)$ & $\Phi_4^2.A_1(q)^4$ & 4 chars. & 96 & 16\\
2 & & 4 & & & 12 chars. & 32 & 14\\ 
3 & & 4 & $G^F$ & $C_{G^*}(s)^F$ & 32 chars. & 1 & 1 \\ \hline
4 & $E_6(q)A_2(q)$ & 4 & $\Phi_4^2.D_4(q)$ & $\Phi_1^2 \Phi_4^2. A_2(q)$ & $\phi_{111}, \phi_{21}, \phi_3$ & 96 & 16\\ 
5 & & 4 & $\Phi_4.^2\!D_6(q)$ & $\Phi_1.\Phi_4.^2\!A_3(q)A_2(q)$ & $\phi_{22} \otimes \phi_{111}$, & 4 & 4  \\
& & 4 & & & $\phi_{22} \otimes \phi_{21}$, & 4 & 4\\
& & 4 & & &  $\phi_{22} \otimes \phi_3$  & 4& 4\\ 
6 & & 4 & $G^F$ & $C_{G^*}(s)^F$ & 30 chars. & 1 & 1 \\ \hline
7 & $^2\!E_6(q)^2\!A_2(q)$ & 4 & $\Phi_4^2.D_4(q)$ & $\Phi_2^2 \Phi_4^2. ^2\!A_2(q)$ & $\phi_{111}, \phi_{21}, \phi_3$ & 96 & 16\\ 
8 & & 4 & $\Phi_4.^2\!D_6(q)$ & $\Phi_2 \Phi_4.A_3(q)^2\!A_2(q)$ & $\phi_{22} \otimes \phi_{111}$, & 4 & 4 \\
& & 4 & & & $\phi_{22} \otimes \phi_{21},$ & 4 & 4\\
& & 4 & & & $\phi_{22} \otimes \phi_3$ & 4 & 4\\
9 & & 4 & $G^F$ & $C_{G^*}(s)^F$ & 30 chars. & 1 & 1 \\  \hline
10 & $D_5(q)A_3(q)$ & 4 & $\Phi_4^3.A_1(q^2)$ & $\Phi_1 \Phi_2 \Phi_4^3$ & 1 & 128 & 56\\ 
11 & & 4 & $\Phi_4^2.D_4(q)$ & $\Phi_1 \Phi_4^2. A_3(q)$ & $\phi_{22}$ & 32 & 14\\ 
12 & & 4 & $\Phi_4^2.D_4(q)$ & $\Phi_2 \Phi_4^2. ^2\!A_3(q)$ & $\phi_{22}$ & 16 & 16\\ 
13 & & 4 & $\Phi_4.^2\!D_6(q)$ & $\Phi_4.^2\!A_3(q)A_3(q)$ & $\phi_{22} \otimes \phi_{22}$ & 4 & 4 \\
14 & & 4 & $\Phi_4.^2\!D_6(q)$ & $\Phi_2 \Phi_4.D_5(q)$ & $(12,04)$,  & 4 & 4 \\
& & 4 & & & $(123,014)$ & 4 & 4\\
15 & & 4 & $G^F$ & $C_{G^*}(s)^F$ & 2 chars. & 1 & 1 \\ \hline
16 & $^2\!D_5(q)^2\!A_3(q)$ & 4 & $\Phi_4^3.A_1(q^2)$ & $\Phi_1 \Phi_2 \Phi_4^3$ & 1 & 128 & 56 \\ 
17 & & 4 & $\Phi_4^2.D_4(q)$ & $\Phi_2 \Phi_4^2.^2\!A_3(q)$ & $\phi_{22}$ & 32 & 14\\
18 & & 4 & $\Phi_4^2.D_4(q)$ & $\Phi_1 \Phi_4^2.A_3(q)$ & $\phi_{22}$ & 16 & 16\\
19 & & 4 & $\Phi_4.^2\!D_6(q)$ & $\Phi_4. ^2\!A_3(q)A_3(q)$ & $\phi_{22} \otimes \phi_{22}$ & 4 & 4\\
20 & & 4 & $\Phi_4.^2\!D_6(q)$ & $\Phi_1 \Phi_4.^2\!D_5(q)$ & $(014,2)$,  &4 & 4\\
& & 4 & & & $(0134,12)$ & 4 & 4\\
21 & & 4 & $G^F$ & $C_{G^*}(s)^F$ & 2 chars. & 1 & 1 \\ \hline
22 & $A_4(q)^2$ & 4 & $\Phi_4^2.A_1(q^2)^2$ & $\Phi_1^2 \Phi_2^2 \Phi_4^2$ & 1 & 16 & 16\\
23 & & 4 & $\Phi_4.^2\!D_6(q)$ & $\Phi_1 \Phi_2 \Phi_4.A_4(q)$ & 6 chars. & 4 & 4 \\
24 & & 4 & $G^F$ & $C_{G^*}(s)^F$ & 9 chars. & 1 & 1 \\ \hline
25 & $^2\!A_4(q)^2$ & 4 & $\Phi_4^2.A_1(q^2)^2$ & $\Phi_1^2 \Phi_2^2 \Phi_4^2$ & 1 & 16 & 16\\
26 & & 4 & $\Phi_4.^2\!D_6(q)$ & $\Phi_1 \Phi_2 \Phi_4.^2\!A_4(q)$ & 6 chars. & 4 & 4\\
27 & & 4 & $G^F$ & $C_{G^*}(s)^F$ & 9 chars. & 1 & 1 \\ \hline
28 & $A_4(q^2)$ & 4 & $\Phi_4^4$ & $\Phi_4^4$ & 1 & 120 & 7\\ \hline
29 & $A_5(q)A_2(q)A_1(q)$ & 4 & $\Phi_4.^2\!D_6(q)$ & $\Phi_1 \Phi_2 \Phi_4.A_2(q)A_1(q)^2$ & 12 chars. & 4 & 4\\ 
30 & & 4 & $G^F$ & $C_{G^*}(s)^F$ & 18 chars. & 1 & 1 \\ \hline
31 & $^2\!A_5(q) ^2\!A_2(q)A_1(q)$ & 4 & $\Phi_4.^2\!D_6(q)$ & $\Phi_1 \Phi_2 \Phi_4.^2\!A_2(q)A_1(q)^2$ & 12 chars. & 4 & 4 \\ 
32 & & 4 & $G^F$ & $C_{G^*}(s)^F$ & 18 chars. & 1 & 1 \\ \hline
33 & $A_7(q)A_1(q)$ & 4 & $\Phi_4^2.D_4(q)$ & $\Phi_1 \Phi_2^2 \Phi_4^2.A_1(q)$ & $\phi_{11}, \phi_2$ & 32 & 14 \\ 
34 & & 4 & $\Phi_4.^2\!D_6(q)$ & $\Phi_1 \Phi_2 \Phi_4.A_3(q)A_1(q)$ & $\phi_{11} \otimes \phi_{22}$, & 4 & 4\\
& & & & &  $\phi_{2} \otimes \phi_{22}$ &  4 & 4\\
35 & & 4 & $G^F$ & $C_{G^*}(s)^F$ & 8 chars. & 1 & 1 \\ \hline
36 & $^2\!A_7(q)A_1(q)$ & 4 & $\Phi_4^2.D_4(q)$ & $\Phi_1^2 \Phi_2 \Phi_4^2.A_1(q)$ & $\phi_{11}, \phi_2$ & 32 & 14\\
37 & & 4 & $\Phi_4.^2\!D_6(q)$ & $\Phi_1 \Phi_2 \Phi_4. ^2\!A_3(q)A_1(q)$ & $\phi_{11} \otimes \phi_{22}$, & 4 & 4\\
& & & & & $\phi_2 \otimes \phi_{22}$ & 4 & 4\\
38 & & 4 & $G^F$ & $C_{G^*}(s)^F$ & 8 chars. & 1 & 1 \\ \hline
39 & $A_8(q)$ & 4 & $\Phi_4^2.A_1(q^2)^2$ & $\Phi_1^2 \Phi_2^2 \Phi_4^2$ & 1 & 32 & 14 \\ 
40 & &  4 & $\Phi_4 ^2\!D_6(q)$ & $\Phi_1 \Phi_2 \Phi_4.A_4(q)$ & $\phi_{2111}, \phi_{311}, \phi_{41}$ & 4 & 4\\ 
41 & & 4 & $G^F$ & $C_{G^*}(s)^F$ & 8 chars. & 1 & 1 \\ \hline
42 & $^2\!A_8(q)$ & 4 & $\Phi_4^2.A_1(q^2)^2$ & $\Phi_1^2 \Phi_2^2 \Phi_4^2$ & 1 & 32 & 14\\ 
43 & & 4 & $\Phi_4.^2\!D_6(q)$ & $\Phi_1 \Phi_2 \Phi_4.^2\!A_4(q)$ & $\phi_{2111}, \phi_{311}, \phi_{41}$ & 4 & 4\\ 
44 & & 4 & $G^F$ & $C_{G^*}(s)^F$ & 8 chars. & 1 & 1 \\ \hline 
45 & $D_8(q)$ & 4 & $\Phi_4^2$ & $\Phi_4$ & 1 & 3072 & 60 \\
46 & & 4 & $\Phi_4^2.D_4(q)$ & $\Phi_4^2.D_4(q)$ & 4 chars. & 32 & 14 \\
47 & & 4 & $G^F$ & $C_{G^*}(s)^F$ & 4 chars. & 1 & 1 \\ \hline \hline
1 & $E_7(q)A_1(q)$ & 5 & $\Phi_5. A_4(q)$ & $\Phi_1 \Phi_5. A_2(q)A_1(q)$ & 6 chars. & 10 & 10 \\ 
2 & & 5 & $G^F$ & $C_{G^*}(s)^F$ & 92 chars. & 1 & 1 \\ \hline
3 & $E_6(q)A_2(q)$ & 5 & $\Phi_5.A_4(q)$ & $\Phi_1 \Phi_5. A_2(q) A_1(q)$ & 6 chars. & 5 & 5 \\ 
4 & & 5 & $G^F$ & $C_{G^*}(s)^F$ & 60 chars.  & 1 & 1 \\ \hline
5 & $D_5(q)A_3(q)$ & 5 & $\Phi_5.A_4(q)$ & $\Phi_1 \Phi_5.A_3(q)$ & 5 chars. & 5 & 5 \\ 
6 & & 5 & $G^F$ & $C_{G^*}(s)^F$ & 75 chars. & 1 & 1 \\ \hline
7 &  $A_4(q)^2$ & 5 & $\Phi_5^2$ & $\Phi_5^2$ & 1 & 25 & 25 \\
8 & & 5 & $\Phi_5. A_4(q)$ &  $\Phi_5.A_4(q)$ & 4 chars. & 5 & 5 \\ 
9 & & 5 & $G^F$ & $C_{G^*}(s)^F$ & chars. & 1 & 1 \\ \hline
10 & $A_5(q)A_2(q)A_1(q)$ & 5 & $\Phi_5.A_4(q)$ & $\Phi_1 \Phi_5.A_2(q) A_1(q)$ & 6 chars. & 5 & 5 \\ 
11 & & 5 & $G^F$ & $C_{G^*}(s)^F$ & 36 chars. & 1 & 1 \\ \hline
12 & $A_7(q)A_1(q)$ & 5 & $\Phi_5.A_4(q)$ & $\Phi_1 \Phi_5. A_2(q)A_1(q)$ & 6 chars. & 5 & 5 \\ 
13 & & 5 & $G^F$ & $C_{G^*}(s)^F$ & 14 chars. & 1 & 1 \\ \hline
14 & $A_8(q)$ & 5 & $\Phi_5.A_4(q)$ & $\Phi_1 \Phi_5. A_3(q)$ & 5 chars. & 5 & 5 \\
15 & & 5 & $G^F$ & $C_{G^*}(s)^F$ & 5 chars. & 1 & 1 \\ \hline
16 & $D_8(q)$ & 5 & $\Phi_5.A_4(q)$ & $\Phi_1 \Phi_5.A_3(q)$ & 5 chars. & 10 & 10 \\
17 & & 5 & $G^F$ & $C_{G^*}(s)^F$ & 70 chars. & 1 & 1 \\ \hline \hline
1 & $E_7(q)A_1(q)$ & 7 &  $\Phi_1 \Phi_7.A_1(q)$ & $\Phi_1 \Phi_7.A_1(q)$ & $\phi_{11} \phi_2$ & 14 & 14 \\ 
2 & & 7 & $G^F$ & $C_{G^*}(s)^F$ & 124 chars. & 1 & 1\\ \hline
3 & $A_7(q)A_1(q)$ & 7 &  $\Phi_1 \Phi_7.A_1(q)$ & $\Phi_1 \Phi_7.A_1(q)$ & $\phi_{11} ,\phi_2$ & 7 & 7 \\ 
4 & & 7 & $G^F$ & $C_{G^*}(s)^F$ & 30 chars. & 1 & 1 \\ \hline
5 & $A_8(q)$ & 7 & $\Phi_1 \Phi_7.A_1(q)$ & $\Phi_1 \Phi_7.A_1(q)$ & $\phi_{11}, \phi_2$ & 7 & 7 \\ 
6 & & 7 & $G^F$ & $C_{G^*}(s)^F$ & 16 chars. & 1 & 1 \\ \hline
7 & $D_8(q)$ & 7 & $\Phi_1 \Phi_7.A_1(q)$ & $\Phi_1^2 \Phi_7$ & 1 & 14 & 14\\ 
8 & & 7 & $G^F$ & $C_{G^*}(s)^F$ & 104 chars. & 1 & 1\\ \hline \hline
1 & $E_7(q)A_1(q)$ & 8 & $\Phi_8. ^2\!D_4(q)$ &  $\Phi_8. A_1(q)^2A_1(q^2)$ & 8 chars. & 8 & 8 \\
2 & & 8 & $G^F$ & $C_{G^*}(s)^F$ & 88 chars. & 1 & 1\\  \hline
3 & $E_6(q)A_2(q)$ & 8 & $\Phi_8.^2\!D_4(q)$ & $\Phi_1 \Phi_2 \Phi_8.A_2(q)$ &  $\phi_{111}, \phi_{21}, \phi_3$ & 8 & 8 \\
4 & & 8 & $G^F$ & $C_{G^*}(s)^F$ & 66 chars. & 1 & 1\\ \hline
5 & $^2\!E_6(q)^2\!A_2(q)$ & 8 & $\Phi_8.^2\!D_4(q)$ & $\Phi_1 \Phi_2 \Phi_8. ^2\!A_2(q)$ & $\phi_{111}, \phi_{21}, \phi_3$ & 8 & 8\\
6 & & 8 & $G^F$ & $C_{G^*}(s)^F$ & 66 chars. & 1 & 1\\ \hline
7 & $D_5(q)A_3(q)$ & 8 & $\Phi_8.^2\!D_4(q)$ & $\Phi_2 \Phi_8.A_3(q)$ & 5 chars. & 8 & 8\\ 
8 & & 8 & $G^F$ & $C_{G^*}(s)^F$ & 60 chars. & 1 & 1 \\ \hline
9 & $^2\!D_5(q)^2\!A_3(q)$ & 8 & $\Phi_8.^2\!D_4(q)$ & $\Phi_1 \Phi_8.^2\!A_3(q)$ & 5 chars. & 8 & 8\\ 
10 & & 8 & $G^F$ & $C_{G^*}(s)^F$ & 60 chars. & 1 & 1\\ \hline
11 & $^2\!A_4(q^2)$ & 8 & $\Phi_8.A_1(q^4)$ & $\Phi_1 \Phi_2 \Phi_4 \Phi_8$ & 1 & 4 & 4\\
12 & & 8 & $G^F$ & $C_{G^*}(s)^F$ & 3 chars. & 1 & 1 \\  \hline \hline
1 & $E_7(q)A_1(q)$ & 9 & $\Phi_9.A_2(q)$ & $\Phi_1 \Phi_9.A_1(q)$ & $\phi_{11}, \phi_2$ & 18 & 18\\ 
2 & & 9 & $G^F$ & $C_{G^*}(s)^F$ & 116 chars. & 1 & 1\\ \hline
3 & $E_6(q)A_2(q)$ & 9 & $\Phi_9.A_2(q)$ & $\Phi_9.A_2(q)$ & $\phi_{111}, \phi_{21}, \phi_3$ & 9 & 9\\ 
4 & & 9 & $G^F$ & $C_{G^*}(s)^F$ & 63 chars. & 1 & 1\\ \hline
5 & $A_8(q)$ & 9 & $\Phi_9.A_2(q)$ & $\Phi_3 \Phi_9$ & 1 & 9 & 9\\
6 & & 9 & $G^F$ & $C_{G^*}(s)^F$ & 21 chars. & 1 & 1\\ \hline \hline
1 & $E_7(q)A_1(q)$ & 12 & $\Phi_{12}. ^3\!D_4(q)$ & $\Phi_{12}.A_1(q) A_1(q^3)$ & 4 chars. & 12 & 12\\ 
2 & & 12 & $G^F$ & $C_{G^*}(s)^F$ & 124 chars. & 1 & 1\\ \hline
3 & $E_6(q)A_2(q)$ & 12 & $\Phi_{12}. ^3\!D_4(q)$ & $\Phi_3 \Phi_{12}. A_2(q)$ & $\phi_{111}, \phi_{21}, \phi_3$ & 12 & 12\\ 
4 & & 12 & $G^F$ & $C_{G^*}(s)^F$ & 54 chars. & 1 & 1\\ \hline
5 & $^2\!E_6(q)^2\!A_2(q)$ & 12 & $\Phi_{12}. ^3\!D_4(q)$ & $\Phi_6 \Phi_{12}. ^2\!A_2(q)$ & $\phi_{111}, \phi_{21}, \phi_3$ & 12 & 12\\ 
6 & & 12 & $G^F$ & $C_{G^*}(s)^F$ & 54 chars. & 1 & 1\\ \hline
7 & $^2\!A_4(q^2)$ & 12 & $\Phi_{12}. ^2\!A_2(q^2)$ & $\Phi_4 \Phi_{12}. A_1(q^2)$  & $\phi_{11}, \phi_2$ & 3 & 3\\ 
8 & & 12 & $G^F$ & $C_{G^*}(s)^F$ & $\phi_{311}$ & 1 & 1\\ \hline
9 & $D_8(q)$ & 12 & $\Phi_{12}. ^2\!A_2(q^2)$ & $\Phi_4 \Phi_{12}. A_1(q^2)$ & $\phi_{11}, \phi_2$ & 12 & 12\\ 
10 & & 12 & $G^F$ & $C_{G^*}(s)^F$ & 96 chars. & 1 & 1\\ \hline \hline
1 & $^2\!A_4(q^2)$ & 20 & $\Phi_{20}$ & $\Phi_{20}$ & 1 & 5 & 5\\ 
2 & & 20 & $G^F$ & $C_{G^*}(s)^F$ & $\phi_{221}, \phi_{32}$ & 1 & 1\\ \hline
\end{longtable}

\normalsize

\begin{proof}
Similar to the proof of Theorem \ref{e-cuspidal pairs F4 good l}.
\end{proof}

\subsection{$e$-cuspidal pairs of $G_2(q)$ and $^3\!D_4(q)$}

Let $G$ be a simple, simply connected of type $G_2$ or $D_4$ defined over $\mathbb{F}_q$ with Frobenius endomorphism $F: G \to G$ such that $G^F=G_2(q)$ or $G^F=\text{}^3\!D_4(q)$. Here, $e$ is relevant for some quasi-isolated $1 \neq s \in G^{*F}$ if and only if $e \in \{1,2,3,6\}$. It remains to determine the $e$-cuspidal pairs for $e=3$.

\begin{thm} \label{Conjecture G2 3D4 good primes}
Let $e=3$. For any quasi-isolated semisimple element $1 \neq s \in G^{*F}$, the $e$-cuspidal pairs $(L, \lambda)$ of $G$ with $\lambda \in \mathcal{E}(L^F,s)$ (up to $G^F$-conjugacy), and the order of their relative Weyl groups $W=W_{G^F}(L, \lambda)$ are as indicated in Table \ref{Table: Quasi-isolated blocks of G_2(q) and 3D4(q) good l}. In particular, generalized $e$-Harish-Chandra theory holds in $\mathcal{E}(G^F,s)$ for every quasi-isolated semisimple element $1 \neq s \in G^{*F}$.
\end{thm}

\begin{longtable}{|c|l|l|c|l|l|l|c|c|}
\caption{Quasi-isolated blocks of $G_2(q)$ and $^3\!D_4(q)$} \\ \hline
\label{Table: Quasi-isolated blocks of G_2(q) and 3D4(q) good l}
No. & $G^F$ & $C_{G^*}(s)^F$  & $e$ & $L^F$ & $C_{L^*}(s)^F$ & $\lambda$  & $|W|$ & $l(B)$ \\ \hline \hline
1 & $G_2(q)$ & $A_2(q)$ & 3 & $\Phi_3$ & $\Phi_3$ & 1 & 3 & 3\\ \hline \hline
2 & $^3\!D_4(q)$ & $A_1(q)A_1(q^3)$ & 3 & $\Phi_1 \Phi_3.A_1(q)$ & $\Phi_1 \Phi_3. A_1(q)$ & $\phi_{11}, \phi_2$ & 2 & 2\\ \hline
\end{longtable}

\begin{proof}
Similar to the proof of Theorem \ref{e-cuspidal pairs F4 good l}.
\end{proof}

\begin{center}\textsc{$^2\!E_6(2)$ and $E_8(2)$}
\end{center} 

Note that these groups do not have semisimple elements of even order. Furthermore, note that the Mackey Formula holds for $e=1$ regardless of $q$ since $1$-split Levi subgroups are contained in $F$-stable parabolic subgroups. In this case, Lusztig induction is just ordinary Harish-Chandra induction. Consequently, the proofs of the previous section still hold for $e=1$ for these groups.

\begin{prop} \label{Exceptions, Mackey not proved}
The assertion of the Theorems \ref{e-cuspidal pairs E6 good l} and \ref{e-cuspidal pairs $E_8$ good l} are still valid when $q=2$.
\end{prop}

\begin{proof}
The only thing missing was to check that  generalized $e$-Harish-Chandra theory holds in every case of the corresponding tables. \\
\textit{$^2\!E_6(2)$}: Since every semisimple element in $^2\!E_6(2)$ has odd order, the only centralizer types that occur are $A_2^3$ and $D_4$. Further, recall that the $e$-cuspidal pairs of $^2\!E_6(2)$ are the Ennola duals of the ones of $E_6(2)$. Let $(L, \lambda)$ be an $e$-cuspidal pair for a semisimple, quasi-isolated  element with centralizer of type $A_2^3$.  From the tables it follows that either $L=G$ or that $\lambda$ is uniform. Hence the decomposition of $R_L^G(\lambda)$ can be determined without using the Mackey formula, so the proof of Theorem \ref{e-cuspidal pairs E6 good l} still works.

Now, let $(L, \lambda)$ be an $e$-cuspidal pair corresponding to a quasi-isolated element $s \in G^{*F}$ with $C_{G^*}(s)^F=\Phi_2^2.D_4(2)^3.3$. If $e=2$ there are two 2-cuspidal pairs $(L_1,\lambda_1)=(\Phi_2^6,1)$ and $(L_2, \lambda_2)=(\Phi_2^2.D_4(2),(02,13))$. Since $\lambda_1$ is uniform, we can decompose $R_{L_1}^G(\lambda_1)$ without using the Mackey formula. For the second pair we use the following argument. We observe that $\pi_{uni}(R_{L_2}^G(\lambda_2)) \in \frac{1}{4} \mathbb{Z}\mathcal{E}(G^F,s)$. Since $R_{L_2}^G(\lambda_2) \in \mathbb{Z}\mathcal{E}(G^F,s)$ is a generalized character, there exists an element $\gamma \in \mathbb{Q}\mathcal{E}(G^F,s)$ which is orthogonal to the space of uniform class functions of $G^F$, such that 
$ \pi_{uni}(R_{L_2}^G(\lambda_2))+ \gamma \in \mathbb{Z}\mathcal{E}(G^F,s)$. Furthermore, we know that $R_{L_1}^G(\lambda_1)$ and $R_{L_2}^G(\lambda_2)$ do not have any irreducible constituents in common because their constituents lie in different blocks by Theorem \ref{blocks good l} (a). In this particular case this already determines the constituents of $\gamma$. Without knowing the norm of $R_{L_2}^G(\lambda_2)$, we are unfortunately not able to determine the multiplicities of the individual constituents. However, it is enough for our purposes to know the constituents. 

A similar argument is needed for $e=3$ (and $e=6)$. There are four 3-cuspidal pairs $(L_i, \lambda_i)$, $i=1, \dots ,4$ with $L:=L_1=L_2=L_3= \Phi_3.^2\!A_2(2)$ and $L_4=G^F$. Again, we are able to determine the constituents of $R_{L_i}^G(\lambda_i)$ for $i=1,2,3$ (the case $i=4$ being trivial). In addition to the arguments used for $e=2$ above, we know that $\lambda_1+\lambda_2+\lambda_3$ is uniform. Therefore, $R_{L}^G( \lambda_1+\lambda_2+\lambda_3)$ is also uniform by transitivity of Lusztig induction (see \cite[11.5 Transitivity]{Digne-Michel}). The same arguments as for $e=2$ yield that a generalized $e$-Harish Chandra theory holds.

For the quasi-isolated elements $s \in G^{*F}$ with $C_{G^*}(s)^F= \Phi_6.^3\!D_4(2).3$ we argue the same way: either $\lambda$ is uniform; $\lambda$ is an $e$-cuspidal character of $G^F$ already; or we can determine the constituents of $R_L^G(\lambda)$ without using the Mackey formula, as for the other 3-cuspidal pairs.

\textit{$E_8(2)$}: The only cases to consider are the ones corresponding to centralizers of type $A_8$, $A_4 \times A_4$ and $E_6 \times A_2$. For every $e$-cuspidal pair $(L, \lambda)$ corresponding to the first or second centraliser type, $\lambda$ is uniform. Hence, we can determine the decomposition of $\mathcal{E}(G^F,s)$ without the Mackey-formula. For the last centraliser type we use the same arguments as for the troublesome cases of $^2\!E_6(2)$. 
\end{proof}

\begin{center}
\textsc{Proof of Theorem \ref{e-Harish-Chandra good prime}}
\end{center}
Let us briefly recall our setting. Let $G$ be a simple, simply connected algebraic group of exceptional type defined over $\mathbb{F}_q$ with Frobenius endomorphism $F: G \to G$, or let $G$ be a simple, simply connected algebraic group of type $D_4$ defined over $\mathbb{F}_q$ with Frobenius endomorphism $F:G \to G$ such that $G^F=\text{} ^3\!D_4(q)$. Let $\ell \nmid q$ be a good prime for $G$ and further assume that $\ell \neq 3$ if $G^F=\text{}^3\!D_4(q)$.

\A*

\begin{proof}
If $s=1$, the assertion follows from \cite{BMM} and \cite{Cabanes-Enguehardunipotent}. If $e=1$ or 2, the assertion follows from \cite{Malle-Kessar}, unless when $G$ is  of type $E_6$ or $E_7$ and $s$ is of order 6. In all other cases the assertion follows from Theorem \ref{e-cuspidal pairs F4 good l}, Theorem \ref{e-cuspidal pairs E6 good l}, Theorem \ref{e-cuspidal pairs $E_7$ good l}, Theorem \ref{e-cuspidal pairs $E_8$ good l}, Theorem \ref{Conjecture G2 3D4 good primes} and Proposition \ref{Exceptions, Mackey not proved}.
\end{proof}

\section{On the Malle--Robinson conjecture}
\noindent
Let $H$ be a finite group. If $N \unlhd K \subseteq H$ are two subgroups of $H$, we call the quotient $K/N$ a \textbf{section} of $H$. The \textbf{sectional $\ell$-rank} $s(H)$ of a finite group $H$ is then  defined to be the maximum of the ranks of elementary abelian $\ell$-sections of $H$. Note that $s(K/N) \leq s(H)$ for every section  $K/N$ of $H$. \\

\noindent
For a block $B$ of $H$ let $l(B) := | \operatorname{IBr}(B)|$. 

\begin{conjecture}[Malle--Robinson, {\cite[Conjecture 1]{Malle-Robinson}}]\label{Malle-Robinson Conjecture}
Let $B$ be an $\ell$-block of a finite group $H$ with defect group $D$. Then 
\begin{align*}
l(B) \leq \ell^{s(D)}.
\end{align*}
\end{conjecture}
\noindent
If strict inequality holds, we say that the conjecture holds in \textit{strong form}. Since the defect groups of a given block $B$ are conjugate and therefore isomorphic to each other, we often write $s(B)$ instead of $s(D)$. 

\begin{definition}\label{Definition: Basic sets}
If $U$ is a union of blocks of $H$ (regarded as a subset of $\operatorname{Irr}(H) \cup \operatorname{IBr}(H)$) we set $\operatorname{Irr}(U)=\bigcup_{B \subseteq U} \operatorname{Irr}(B)$ and $\operatorname{IBr}(U)=\bigcup_{B \subseteq U} \operatorname{IBr}(B)$. 

A subset $A \subseteq \mathbb{Z}\operatorname{IBr}(U)$ is called a \textbf{generating set} for $U$ if it generates
 $\mathbb{Z}\operatorname{IBr}(U)$ as a $\mathbb{Z}$-module, and it is called a \textbf{basic set} for $U$ if it is a basis of $\mathbb{Z}\operatorname{IBr}(U)$ as a $\mathbb{Z}$-module. Let $\hat{H}=\{h \in H \mid \ell \nmid o(h)\}$ denote the set of $\ell$-regular elements of $H$. A subset $C \subseteq \operatorname{Irr}(G)$ is called an \textbf{ordinary} generating (respectively basic) set for $U$ if the set $\hat{C}= \{\hat{\psi} \text{ } | \text{ } \psi \in C\}$ consisting of the restrictions of the irreducible characters in $C$ to $\hat{H}$ is a generating (respectively basic) set for $U$. 
\end{definition}
\noindent
We return to our initial setting. Let $G$ be a connected reductive group defined over $\mathbb{F}_q$ with Frobenius endomorphism $F:G \to G$. Let $\ell \nmid q$ be a good prime for $G$. 

By Theorem \ref{Lusztigseries blocks}, we know that $\mathcal{E}_\ell(G^F,s)$ is a union of $\ell$-blocks of $G^F$. However, we can say even more about $\mathcal{E}_\ell(G^F,s)$. 

\begin{thm}[{\cite[Theorem A]{Geck2}}] \label{Geck}
Assume that $\ell$ is a good prime for $G$ not dividing the order of $(Z(G)/Z^\circ(G))_F$ (the largest quotient of $Z(G)$ on which $F$ acts trivially). Let $s \in G^{*F}$ be a semisimple $\ell'$-element. Then $\mathcal{E}(G^F,s)$ is an ordinary basic set for the union of blocks $\mathcal{E}_\ell(G^F,s)$.
\end{thm}
\noindent
\begin{remark}
Let $B$ be an $\ell$-block contained in $\mathcal{E}_\ell(G^F,s)$ for some semisimple $\ell'$-element $s \in G^{*F}$. It follows that an ordinary basic set for $B$ is then given by $\operatorname{Irr}(B) \cap \mathcal{E}(G^F,s)$. 
\end{remark}

\begin{definition}
(a) The $\ell$-blocks contained in $\mathcal{E}_\ell(G^F,s)$  for a semisimple, quasi-isolated $\ell'$-element $s \in G^{*F}$ are  called \textbf{quasi-isolated}. Further, if $s=1$ they are also called \textbf{unipotent}.\\
(b) Let $H=G^F/Z$, for some subgroup $Z \subseteq Z(G^F)$. A block of $H$ is said to be \textbf{quasi-isolated} if it is dominated by a quasi-isolated block of $G^F$ and \textbf{unipotent} if is dominated by a unipotent block of $G^F$.
\end{definition}

\B*

\begin{proof}
We start with the unipotent blocks. Generalized $e$-Harish-Chandra theory holds in $\mathcal{E}(G^F,1)$ by \cite{BMM}. Let $B=b_{G^F}(L, \lambda)$ be a unipotent block. By the results in Section 3 and Theorem \ref{Geck}, we conclude that $\mathcal{E}(G^F,(L, \lambda))=\mathcal{E}(G^F,s) \cap \operatorname{Irr}(B)$ is a basic set for $B$. This proves the first part of the assertion and the second part was already proved in \cite[Proposition 6.10]{Malle-Robinson}.
 
Now suppose that $B=b_{G^F}(L, \lambda)$ is a non-unipotent, quasi-isolated block of $G^F$. Since we proved that generalized $e$-Harish-Chandra theory holds in every Lusztig series associated to semisimple quasi-isolated elements of $G^{*F}$, the first part of the assertion follows by Section 2 and Theorem \ref{Geck} again. In particular, $l(B)=|\mathcal{E}(G^F,(L, \lambda))|$. These cardinalities can be found in the last column of the tables in Section 3 (see the Tables \ref{table:Quasi-isolated blocks F4 good l}, \ref{Table:Kessar-Malle E_6 good l}, \ref{Table: Quasi-isolated blocks of E_7(q) good l}, \ref{Table: Quasi-isolated blocks of E_8(q) good l} and \ref{Table: Quasi-isolated blocks of G_2(q) and 3D4(q) good l}). Let $D$ be a defect group of $B$. By \cite[Lemma 4.13]{Cabanes-Enguehardtwisted}, $Z(L)_\ell^F \subseteq D$ 
and therefore $s(Z(L)_\ell^F) \leq s(D)$. The structure of $Z(L)_\ell^F$ can be read off from the tables in Section 3. We see that $l(B) \leq \ell^{s(Z(L)_\ell^F)}$ in every case. With this, the second part of the assertion is proved for the quasi-isolated $\ell$-blocks of $G^F$. 

Suppose that $\bar{B}$ is a quasi-isolated block of $H=G^F/Z(G^F)$ with defect group $\bar{D}$. Let $B$ be the quasi-isolated block of $G^F$ that dominates $\bar{B}$. The order of $Z(G^F)$ is either 1 or a bad prime for $G$. Thus, $Z(G^F)$ is an $\ell'$-subgroup by our assumption on $\ell$. By \cite[(9.9) Theorem]{Navarro}, $l(\bar{B})=l(B)$ and $\bar{D}$ is of the form $DZ(G^F)/Z(G^F)$, for a defect group $D$ of $B$. Since $D$ and $Z(G^F)$ commute and $D \cap Z(G^F)=\{1\}$, $DZ(G^F)$ is a direct product. It follows that $D Z(G^F)/Z(G^F) \cong D$, i.e. $s(\bar{D})=s(D)$. Thus, the Malle--Robinson conjecture holds for $\bar{B}$ since it holds for $B$.
\end{proof}

The reason we focused on the quasi-isolated blocks are the results of Bonnafé--Rouquier \cite{Bonnafe-Rouquier} and more recently Bonnafé--Dat--Rouquier \cite{BDR}. Their results play a key role in the proof of Corollary \ref{minimal}.

\begin{definition}
Let $H$ be a finite group and let $B$ be an $\ell$-block of $H$. Then $(H,B)$ (or just $B$, if $H$ is understood) is called a \textbf{minimal counterexample} to the Malle--Robinson conjecture if 
\begin{enumerate}
\item the conjecture does not hold for $B$, and
\item the conjecture holds for all $\ell$-blocks $B'$ of groups $K$ with $|K/Z(K)|$ strictly smaller than $|H/Z(H)|$ having defect groups isomorphic to those of $B$.
\end{enumerate} 
\end{definition}

\C*

\begin{proof}
Suppose that $(H,B)$ is a minimal counterexample to the Malle--Robinson conjecture. Let $D$ be a defect group of $B$. By \cite[Proposition 6.4]{Malle-Robinson}, $H$ is not an exceptional covering group of a finite group of exceptional Lie type. By \cite[Proposition 6.5]{Malle-Robinson}, $H$ is not of Lie type $^2\!B_2$, $^2\!G_2$, $G_2$, $^3\!D_4$ or $^2\!F_4$. Hence, $H=G^F/Z$, where $G$ is a simple, simply connected group of exceptional type ($F_4,E_6,E_7$ or $E_8$) and $Z \subseteq Z(G^F)$ is a central subgroup. By \cite[Proposition 6.1]{Malle-Robinson}, $\ell$ does not divide $q$. Let $B'$ be the unique block of $G^F$ dominating $B$ and let $D'$ be a defect group of $B'$. Recall from the proof of Theorem \ref{Conjecture good} that $l(B)=l(B')$ and $s(D)=s(D')$. By \cite[Theorem 7.7]{BDR}, $B'$ is Morita equivalent to an $\ell$-block $b$ of a subgroup $N$ of $G^F$ and their defect groups are isomorphic. In particular, $l(B')= l(b)$ and $s(B')=s(b)$. If $s$ is not quasi-isolated, then $N$ is a proper subgroup. By the minimality of $(H,B)$, $B$ is therefore a quasi-isolated block of $H$. So, the assertion follows from Theorem \ref{Conjecture good}.
\end{proof}

\end{document}